%% file: matroidkmin.tex
\documentclass[a4paper,11pt]{article}
\usepackage[margin=26mm,nohead]{geometry}

\usepackage{amsmath,amssymb,graphicx}

\usepackage{latexsym,amssymb,graphics,graphicx,epsfig,color,enumerate,ifthen}
\usepackage{xspace}
\usepackage{amsthm}

\usepackage{t1enc}
\usepackage[latin2]{inputenc}

\input{pszkod.ltx}

\newcommand{\cB}{\ensuremath{\mathcal{B}}}
\newcommand{\cH}{\ensuremath{\mathcal{H}}}
\newcommand{\cL}{\ensuremath{\mathcal{L}}}

\newcommand{\cX}{\ensuremath{\mathcal{X}}}
\newcommand{\cY}{\ensuremath{\mathcal{Y}}}
\newcommand{\cZ}{\ensuremath{\mathcal{Z}}}
\newcommand{\cP}{\ensuremath{\mathcal{P}}}
\newcommand{\cM}{\ensuremath{\mathcal{M}}}

\newcommand{\fixme}[1]{\textbf{FIX ME!!! #1}}

\newcommand{\Rset}{\mathbb{R}}
\newcommand{\Zset}{\mathbb{Z}}

 \newtheorem{thm}{Theorem }
 \newtheorem{problem}{Problem }
 
 \newtheorem{claim}[thm]{Claim }
 \newtheorem{cor}[thm]{Corollary }
 \newtheorem{lemma}[thm]{Lemma}

\newboolean{sodaversion}
\setboolean{sodaversion}{false}

\title{Blocking optimal $k$-arborescences}
\author{Attila Bern\'ath\thanks{
MTA-ELTE Egerv\'ary Research Group,
Department of Operations Research, E\"otv\"os University, P\'azm\'any P\'eter s\'et\'any 1/C, Budapest, Hungary, H-1117.  Supported by  the Hungarian Scientific Research Fund (OTKA, grant number K109240). The second author is supported by the MTA Bolyai Research Scholarship.
E-mails: {\tt bernath@cs.elte.hu} (Attila Bern\'ath), {\tt tkiraly@cs.elte.hu} (Tam\'as Kir\'aly).}
\and
Tam\'as Kir\'aly\footnotemark[1]
}

\begin{document}

\maketitle

\newcommand{\karb}{$k$-arborescence}
\newcommand{\krarb}[1][s]{$#1$-rooted $k$-arborescence}

\begin{abstract}
Given a digraph $D=(V,A)$ and a positive integer $k$, an arc set
$F\subseteq A$ is called a \textbf{\karb} if it is the disjoint union
of $k$ spanning arborescences. The problem of finding a minimum cost \karb\ is known to be polynomial-time solvable using matroid intersection.
In this paper we study the following problem: find a minimum cardinality subset of
arcs that contains at least one arc from every minimum cost \karb.
For $k=1$, the problem was solved in [A. Bernáth, G. Pap , Blocking optimal arborescences, IPCO 2013].
In this paper we give an algorithm for general $k$ that has polynomial running time if $k$ is fixed.
\end{abstract}

\begin{quote}
\end{quote}
\vspace{5mm}



\section{Introduction}

The \emph{cuts} of a matroid are the minimal transversals of the family of bases; in other words, a subset of the elements is a cut if it is an inclusionwise minimal subset that contains at least
one element from each base. The problem of finding minimum cuts in matroids has been studied in several different contexts (note the distinction between \emph{minimal} and \emph{minimum}: minimal is shorthand for \emph{inclusionwise minimal}, while minimum means \emph{minimum size}). Perhaps the best known special case is
the minimum cut problem in graphs, which can be solved using network flows, and faster algorithms have also been developed (e.g.\ the Nagamochi-Ibaraki algorithm \cite{nagamochi1992computing}). More generally, the minimum cut of $kM$, where
$M$ is a graphic matroid (or even a hypergraphic matroid, see \cite{egresqp-09-05}), can be found in polynomial time. A notable open question is the complexity of finding
a minimum cut in a rigidity matroid.

\ifsodaversion
The minimum cut problem in binary matroids is NP-complete, as proved by Vardy \cite{vardy1997intractability}; Geelen, Gerards, and Whittle \cite{geelen2013highly} conjecture
that the problem is in P for any minor-closed proper subclass of binary matroids. Partial results in this direction have been achieved by Geelen and Kapadia \cite{geelen2015computing}.

\else
The minimum cut of a transversal matroid can also be found in polynomial time; however, the problem of finding a minimum \emph{circuit}
of a transversal matroid is NP-complete \cite{mccormick1983combinatorial}, which implies that the minimum cut problem is NP-complete for gammoids.
Another line of research considers the problem for binary matroids. NP-completeness was proved by Vardy \cite{vardy1997intractability}; Geelen, Gerards, and Whittle \cite{geelen2013highly} conjecture
that the problem is in P for any minor-closed proper subclass of binary matroids. Partial results in this direction have been achieved by Geelen and Kapadia \cite{geelen2015computing}.
\fi

If we consider \emph{minimum cost bases} (or \emph{optimal bases} for brevity) of a matroid $M$, then these form the bases of another matroid which can be obtained by taking the direct sum of certain minors of $M$. Thus we can find a minimum transversal of the family of optimal bases of $M$ by solving minimum cut problems in some minors of $M$. In particular, if the minimum cut problem is solvable in polynomial time in a minor-closed class
of matroids, then a minimum transversal of optimal bases can also be found in polynomial time in this class. For example, since the class of graphic matroids is minor-closed and the minimum cut problem can be solved efficiently, we can also efficiently find a minimum transversal of optimal spanning trees in a graph with edge costs.

Our paper belongs to a line of research that considers directed versions of this problem. Let $D=(V,A)$ be a digraph with node set $V$ and arc set $A$. A
\textbf{spanning arborescence} is an arc set $F\subseteq A$ that is a
spanning tree in the undirected sense and every node has in-degree at
most one. Thus there is exactly one node, the \textbf{root node}, with
in-degree zero. If the node set is clear from the context, spanning
arborescences will be called \textbf{arborescences} for
brevity. Arborescences can be considered as common bases of two
matroids, so the problem of finding a minimum transversal of the
family of arborescences is a special case of the minimum transversal
problem for common bases of two matroids. This problem  is  NP-hard
 in general (as mentioned above, it is NP-hard even when the two matroids coincide). However, the special case for arborescences  can be formulated as the minimization of the
sum of the in-degrees of two disjoint node sets of the digraph, which can be solved efficiently
using network flows. The problem of finding a minimum transversal of
the family of \emph{minimum cost arborescences} is considerably more
difficult. It can still be solved in polynomial time as shown in
\cite{mincostarb}, but the solution requires more sophisticated tools than network flows.

The arc-disjoint union of $k$ spanning arborescences is called a \textbf{\karb}. If $F\subseteq A$ is a \karb\ in a digraph $D=(V, A)$, then its \textbf{root vector} is the vector $q\in \Zset_+^V$ for which $q(v)$ counts the number of arborescences in $F$ that are rooted at $v\in V$. Note that the root vector is determined by the in-degrees, as $q(v)=k-\varrho_F(v)$ for every $v\in V$, so it does not depend on the way a \karb\ is decomposed into arborescences. If every
arborescence has the same root node $s$, then $F$ is called an \textbf{\krarb}.
Given $D=(V,A)$, $k$ and a cost function $c:A\to \Rset_+$, a
{minimum cost \karb} or a {minimum cost \krarb} can be found efficiently using the matroid intersection algorithm;
see \cite[Chapter 53.8]{Schrijver} for a reference, where several related problems are considered. The existence of an \krarb\ is characterized by Edmonds' disjoint arborescence theorem,
while the existence of a \karb\ is characterized by a theorem of Frank \cite{frank1978disjoint}. Frank also gave a linear programming description of the convex hull of \karb{s}, generalizing Edmonds' linear programming description of the convex hull of \krarb{s}.

In this paper we consider the following two problems.

\begin{problem}[\textbf{Blocking optimal \karb s}]\label{prob:1}
Given a digraph $D=(V,A)$, a positive integer $k$, and a cost
function $c:A\to \Rset_+$, find a minimum cardinality transversal of the family of
minimum cost \karb s.
\end{problem}

\begin{problem}[\textbf{Blocking optimal \krarb s}]\label{prob:2}
Given a digraph $D=(V,A)$, a node $s \in V$, a positive integer $k$, and a cost
function $c:A\to \Rset_+$, find a minimum cardinality transversal of the family of
minimum cost \krarb s.
\end{problem}

In Section \ref{sec:variants} we show that the two problems are
polynomial-time equivalent. For $k=1$, these problems have been solved in
\cite{mincostarb}. Moreover, Problem \ref{prob:1} is solved in
\cite{kmincostarb} in the special case when $c\equiv 1$ (note that Problem \ref{prob:2} is a minimum cut problem when $c\equiv 1$). The papers
\cite{mincostarb,kmincostarb} also consider more general weighted versions of these problems.

The main result of the present paper is an algorithm for Problems \ref{prob:1} and \ref{prob:2}  that has polynomial running time when $k$ is constant.
It remains open whether there is a polynomial-time algorithm when $k$ is not fixed, or indeed whether there is an FPT algorithm where $k$ is the parameter.
Along the way we obtain the following result of independent interest: the convex hull of root vectors of minimum cost \karb s is a base polyhedron.
This generalizes the result of Frank \cite{frank1978disjoint} stating that the root vectors of \karb s form a base polyhedron.

The paper is organized as follows. After a brief section on notation, the relationship between different versions of the problem is discussed in Section \ref{sec:variants}, including a dual characterization of optimal \karb{s}.
The next section describes the \textbf{matroid-restricted \karb } problem, a generalization of \karb{s} introduced by Frank \cite{frank2009} that is essential to the proof of the main result. In Section \ref{sec:Ltight}, we describe the connection between matroid-restricted \karb{s} and the dual characterization of optimal \karb{s}. A corollary of this connection is that the
convex hull of the root vectors of optimal \karb{s} is a
base polyhedron (Theorem \ref{thm:optimalroot}).

The structure of minimal transversals is analyzed in Section \ref{sec:blocking}. In the case when the size of the minimum transversal is at least $k$, we derive that there is a minimum transversal with a special structure (Theorem \ref{thm:main}). This leads to the main result of the paper, an algorithm that finds a minimum transversal of optimal \karb{s} in polynomial time if $k$ is constant.

\subsection{Notation}

Let us overview some of the notation and definitions used in the
paper.   Given a digraph $D=(V,A)$ and a node set $Z\subseteq V$,
let $D[Z]$ be the subdigraph induced by $Z$. If $E\subseteq A$ is a
subset of the arc set, then we will identify $E$ and the subgraph
$(V,E)$. Thus $E[Z]$ is obtained from $(V,E)$ by deleting the nodes of
$V-Z$. The arc set of the digraph $D$ will also be denoted
by $A(D)$. The set of arcs of $D$
entering a node set $Z$ is denoted $\delta_D^{in}(Z)$, and $\varrho_D(Z)=|\delta_D^{in}(Z)|$.  For an undirected or directed graph $G=(V,E)$ and a
subset $X\subseteq V$, $i_G(X)$ denotes the number of edges with
both endpoints in $X$.

A \textbf{subpartition} of a subset $X$ of $V$ is a collection of
pairwise disjoint non-empty subsets of $X$. Note that $\emptyset$
cannot be a member of a subpartition, but $\emptyset$ is a valid
subpartition, having no members at all. A set family $\cL\subseteq
2^V$ is said to be \textbf{laminar} if any two members of \cL\ are either disjoint, or one contains the other.
For a vector $x:A\to \Rset$ and subset $Z\subseteq A$ we use the
notation $x(Z)=\sum_{a\in Z}x_a$.

In the paper we will use the $-$ (minus) operator in many roles
beyond subtraction of numbers: for example we will use it for
set-theoretical difference instead of $\setminus$. Furthermore, for a
digraph $D=(V,A)$ and $E\subseteq A$ we will use the notation $D-E$ to mean the digraph $(V,
A-E)$. A one-element set
$\{e\}$ will be denoted without braces by $e$ in some contexts;
for example, $E-e$ means $E-\{e\}$, and this is used even if $e\notin
E$, in which case $E-e=E$. Similarly, for a subpartition $\cX$ and for a member
$X\in \cX$, we write $\cX-X$ instead of $\cX-\{X\}$.

For general background on \textbf{matroids} and \textbf{base polyhedra} we refer the reader to \cite{frank2011connections}. Given a
matroid $M=(S,r)$ (where $S$ is the ground set and $r$ is the rank function) and a positive integer $k$, the \textbf{$k$-shortening} of
$M$ is the matroid $(S, r')$ where $r'(E)=\min \{r(E), k\}$.

Given a function $p:2^S\to \Rset$, a subset $X\subseteq S$ is called
\textbf{separable} if there exists a partition $X_1, X_2, \dots, X_t$
of $X$ such that $p(X)\le \sum_i p(X_i)$. The function $p$ is called
\textbf{near supermodular} if $p(X)+p(Y)\le p(X\cap Y) + p(X\cup Y)$
holds for every intersecting pair $X,Y\subseteq V$ of non-separable
sets. The \textbf{(upper) truncation} of a set function $p:2^S\to \Rset$ (satisfying $p(\emptyset)=0$) is a
set function $p^\wedge:2^S\to \Rset$ defined by
\[p^\wedge(X)=\max\{\sum \{p(Z): Z\in \cZ\}: \cZ \mbox{ is a partition of }X\}.
\]

\begin{thm}\cite[Theorems 15.1.1 and 15.1.3]{frank2011connections}\label{thm:wedge}
The truncation of a near supermodular function is fully
supermodular. The truncation of a nonnegative function is monotone increasing.
If $p$ is near supermodular and the polyhedron $B(p)=\{x \in \Rset^S: x(S)=p(S),\ x(Z) \geq p(Z)\ \forall Z \subseteq S\}$
is non-empty, then $B(p)$ is a base polyhedron and $B(p)=B(p^{\wedge})$.
\end{thm}

Given a digraph $D=(V,A)$ and a positive integer $\alpha$, we will often use an extended digraph
$D^+=(V+s,A^+)$, called the \textbf{$\alpha$-extension} of $D$, that has a new  node $s\notin V$ and $\alpha$ parallel
arcs from $s$ to every node in $V$. If a cost function $c:A\to \Rset$ is also given, then we
extend $c$ to a function $c^+:A^+\to \Rset$ so that $c^+(uv)=c(uv)$
for any $uv\in A$ and $c^+(sv)=\beta$ for any new arc $sv\in A^+-A$, where $\beta$ is some nonnegative real number.
The weighted digraph $(D^+, c^+)$ is then
called the \textbf{$(\alpha,\beta)$-extension of $(D,c)$}.

\section{Relationship between different versions of the problem}\label{sec:variants}

\begin{thm}\label{thm:versions}
Problem \ref{prob:1} (Blocking optimal \karb s) and Problem \ref{prob:2} (Blocking optimal \krarb s) are polynomial-time equivalent.
\end{thm}
\ifsodaversion
The proof of Theorem \ref{thm:versions} can be found in the Appendix.
\else
\begin{proof}
Problem \ref{prob:2} reduces to Probem \ref{prob:1} by deleting all arcs entering node $s$ from the input digraph. For the other direction, consider an instance $D,k,c$ of Problem \ref{prob:1}, and let $\alpha=|A|+k$, $\beta=\sum_{a\in A}c(a)+1$. Let $(D^+,c^+)$ be the $(\alpha,\beta)$-extension of $(D,c)$. In the instance of Problem \ref{prob:2}
given by $(D^+,k,c^+,s)$, the minimum cost \krarb s
naturally correspond to minimum $c$-cost \karb s in $D$
(since they contain exactly $k$ arcs leaving $s$ because of the value of $\beta$). Moreover, the minimum size of a transversal is at most $|A|$ as $A$ itself is a transversal.
This shows that every minimum transversal is a subset of $A$.
\end{proof}
\fi
To describe the structure of minimum cost \karb s, we introduce the notion of a \karb\ being tight for some laminar family of node subsets.
Given a digraph $D=(V,A)$ and a laminar family $\cL\subseteq 2^V$, a \karb\ $F\subseteq A$ is called \textbf{\cL-tight}
if $F[W]$ is a \karb\ in $D[W]$ for every $W\in \cL$.
Note that if $\cL\subseteq 2^{V-s}$, then an \krarb\ $F\subseteq A$ is  $\cL$-tight if and only if  $\varrho_F(W)= k$ for every $W\in \cL$.
The link between $\cL$-tight \krarb s and minimum cost  \krarb s is provided by the following theorem.

\begin{thm}\cite[Corollary 53.6a]{Schrijver}\label{thm:karbhull}
Given a digraph $D=(V,A)$ and a node $s\in V$, the system \eqref{eq:karblin1}--\eqref{eq:karblin2} below is TDI, and it describes the convex hull of subsets of $A$ containing an \krarb[s].
\begin{eqnarray}
0\le x(a)\le 1\mbox{ for every }a\in A \label{eq:karblin1} \\
\varrho_x(Z)\ge k\mbox{ for ever non-empty }Z\subseteq V-s.\label{eq:karblin2}
\end{eqnarray}
If a cost function $c:A\to \Rset$ is also given and we consider the problem of minimizing $cx$ under the conditions above, then there is an optimal dual solution where the dual variables corresponding to \eqref{eq:karblin2} have laminar support.
\end{thm}

Complementary slackness conditions imply the following.
\ifsodaversion
The proof of Corollary \ref{cor:slackness} can be found in the Appendix.
\fi

\begin{cor}\label{cor:slackness}
Given a digraph $D=(V,A)$, a cost function $c:A\to \Rset_+$, a node
$s\in V$ and a positive integer $k$, one can find a laminar family
$\cL\subseteq 2^{V-s}$ and two disjoint arc-sets $A_0, A_1\subseteq A$
with the property that  an \krarb\ $F\subseteq A$ has minimum cost if and only if $A_1\subseteq F\subseteq A-A_0$ and $F$ is \cL-tight.
\end{cor}

\ifsodaversion

The arcs in $A_0$ are called \textbf{forbidden} arcs, while the arcs in $A_1$ are the \textbf{mandatory} arcs.

\else

\begin{proof}
Consider the LP $\min\{cx: x\in \Rset^A,\ 0\le x\le 1,\ \varrho_x(Z)\ge
k$ for ever non-empty $Z\subseteq V-s\}$. By Theorem
\ref{thm:karbhull}, this has an integer optimal solution, which is a
minimum cost \krarb. Let $y^*,z^*$ be an optimal solution of the
dual
\begin{align*}\label{lp:dual}
\max \sum_{\emptyset \neq Z \subseteq V-s} ky_Z-\sum_{a\in A}z_a\\
 y\in \Rset_+^{2^{V-s}-\{\emptyset\}},z\in \Rset_+^A  \\
 \sum_{Z:
  a\in \delta^{in}(Z)}y_Z-z_a \le c_a \text{ for every } a\in A.
\end{align*}
  We can assume that the support of $y^*$ is a laminar family
$\cL\subseteq 2^V$ by Theorem \ref{thm:karbhull}. The complementary slackness conditions show that a
feasible primal solution $x^*$ is optimal if and only if the following
three conditions hold.
\begin{enumerate}
\item $x^*_a=0$ for every $a\in A$ with $\sum_{Z: a\in \delta^{in}(Z)}y^*_Z-z^*_a< c_a$ (\textbf{forbidden} arcs),
\item $\varrho_{x^*}(W)=k$ for every $W\in \cL$, and
\item $x^*_a=1$ for every $a\in A$ with $z^*_a>0$ (\textbf{mandatory} arcs).
\end{enumerate}

By denoting the forbidden arcs by $A_0$ and the mandatory arcs by $A_1$ we obtain the required structure.
\end{proof}

\fi

\begin{thm}
Problem \ref{prob:2} can be reduced to the following Problem \ref{prob:lam} in polynomial time.
\begin{problem}\label{prob:lam}
Given a digraph $D=(V,A)$, a root $s$, and a
laminar family $\cL\subseteq 2^{V-s}$, find a minimum cardinality transversal of the family of \cL-tight \krarb s.
\end{problem}
\end{thm}


\begin{proof}
Given a digraph
$D=(V,A)$, a cost function $c:A\to \Rset_+$, a node $s\in V$ and a positive integer $k$,
we consider $A_0,A_1$, and $\cL$ as in Corollary \ref{cor:slackness}.
If there exists a mandatory arc, then it is a singleton transversal of the family of optimal \krarb s. If $A_1=\emptyset$, then the
problem is equivalent to finding a minimum transversal of the family of $\cL$-tight \krarb s in $A - A_0$.
\end{proof}

Note that we can decide in polynomial time whether an \cL-tight \krarb\ exists by finding a minimum cost \krarb\ for
the cost function $c(e)=|\{W \in \cL: e \in \delta^{in}_D(W)\}|$.

\newcommand{\indep}{matroid-restricted}

\section{Matroid-restricted \karb{s}} \label{sec:matroidal}

In this section we introduce \indep\ \karb s, a notion that will be useful in describing the structure of \cL-tight \karb s.
Let $D=(V,A)$ be a digraph, and for every $v\in V$ let $M_v=(\delta^{in}_D(v), r_v)$ be a matroid. Let furthermore $\cM=\{M_v: v\in V\}$ be the family of these matroids.
A \karb\ $F\subseteq A$ is said to
be \textbf{\cM-\indep} (or \textbf{\indep} for short) if $F\cap \delta^{in}_D(v)$ is independent in $M_v$ for
every $v\in V$.
Similarly, an \krarb\ $F\subseteq A$ is said to be \cM-\indep\ if  $F\cap \delta^{in}_D(v)$ is independent for
every $v\in V-s$ (note that the matroid $M_s$ does not play a role here).
The notion of \indep\ \krarb\  was introduced by Frank \cite{frank2009} in a slightly more general setting,
where there is an additional matroid on the set of arcs leaving  $s$.
Our definition corresponds to the case where this is a free matroid. Some of the results of
this section could be derived from \cite[Theorem 4.5]{frank2009}; however,
since the context is different, it is easier to include self-contained proofs.

Let us define the matroid $M^\oplus=(A, r^\oplus)$ as the
direct sum of the matroids $M_v$ $ (v\in V)$. The following theorem is an easy consequence of the matroid intersection theorem.
\ifsodaversion
The proof of Theorem \ref{thm:indepkarb} can be found in the Appendix.
\fi

\begin{thm}\label{thm:indepkarb}
Given a digraph $D=(V,A)$ and matroids $M_v=(\delta^{in}_D(v), r_v)$
for every $v\in V$, there exists a \indep\ \karb\
in $D$ if and only if the following inequality holds for every subpartition \cX\ of $V$:
\begin{equation}\label{eq:subpart}
\sum \{r^\oplus(\delta^{in}_{D}(X)): X\in \cX\} \ge k(|\cX|-1).
\end{equation}
\end{thm}

\ifsodaversion

\else

\begin{proof}[Proof of Theorem \ref{thm:indepkarb}]
The necessity of \eqref{eq:subpart} is clear: if $F\subseteq A$ is a
\indep\ \karb\ and \cX\ is a subpartition of $V$, then $k(|\cX|-1)\le
\sum_{X\in \cX} \varrho_F(X)\le \sum_{X\in
  \cX}r^\oplus(\delta^{in}_{D}(X))$.  In order to prove sufficiency,
let $M_1=(A,r_1)$ be  $k$ times the circuit matroid of the underlying undirected graph of $D$. Note that condition \eqref{eq:subpart}
implies that $D$ contains $k$ edge-disjoint spanning trees, thus
$r_1(A)=k(|V|-1)$. For every $v\in V$, let $M_v'=(\delta^{in}_D(v),
r'_v)$ be the $k$-shortening of $M_v$, that is $r'_v(E)=\min\{r_v(E),
k\}$ for every $E\subseteq \delta^{in}_D(v)$. Let furthermore $M_2=(A,
r_2)$ be the direct sum of the matroids
$M'_v$. Observe that $F\subseteq A$ is a \indep\ \karb\ in $D$ if and
only if $F$ is a common independent set of $M_1$ and $M_2$ and has
size $k(|V|-1)$. By Edmonds' matroid intersection theorem \cite{edmonds},
such an $F$ exists if and only if
\begin{equation}\label{eq:matroidint}
r_1(E)+ r_2(A-E)\ge k(|V|-1)\mbox{ for every } E \subseteq A.
\end{equation}
We show that condition \eqref{eq:subpart} implies
\eqref{eq:matroidint}. Suppose that \eqref{eq:matroidint} fails for
some $E$. Clearly, we can assume that $E$ is closed in $M_1$ and
$M_1|E$ does not contain bridges (a \textbf{bridge} in a matroid is an element
that is contained in every base).
\begin{claim}
If $E\subseteq A$ is closed in $M_1$ and $M_1|E$ does not contain
bridges, then there exists a partition $\cY$ of $V$
such that $r_1(D[Y])=k(|Y|-1)$ for every $Y\in \cY$ and
$E=\cup_{Y\in \cY}D[Y]$.
\end{claim}
\begin{proof}
We say that a non-empty $Y\subseteq V$ is \textbf{tight} (with respect
to $E$) if $r_1(E[Y])= k(|Y|-1)$. In other words, $Y$ is tight if
$E[Y]$ contains $k$ edge-disjoint  trees, each spanning $Y$. For example, sets of
size 1 are tight. If $Y_1, Y_2$ are both tight and $Y_1\cap Y_2\ne
\emptyset$ then $Y_1\cup Y_2$ is tight, too. To prove this, let
$T_1\subseteq E$ be a tree spanning $Y_1$ and $T_2\subseteq E$ be a
tree spanning $Y_2$, and observe that $T_1$ can be extended to a tree
spanning $Y_1\cup Y_2$ using the edges of $T_2-E[Y_1]$. Therefore let
$\cY$ be the partition of $V$ consisting of the maximal tight
sets. Since $E$ is closed in $M_1$, it contains every arc of $D$ that
is induced in some $Y\in \cY$. Let $G'=(V', E')$ be the graph
obtained from $(V,E)$ after contracting every $Y\in \cY$ into a node
$y$. We claim
that $i_{G'}(Z)<k(|Z|-1)$ for every $Z\subseteq V'$ with $|Z|\ge
2$. Assume not and take an inclusionwise minimal set
$Z$ with $i_{G'}(Z) \geq k(|Z|-1)$. Then $G'[Z]$ contains $k$ edge-disjoint spanning trees by the theorem of Tutte and Nash-Williams \cite{tutte1961problem},
which contradicts the maximality of the tight sets in
$\cY$. This implies that the bases of $M_1|E$ contain every arc of $E$
going between different members of the partition $\cY$. But since
$M_1|E$ does not contain bridges, $E=\cup_{Y\in \cY}D[Y]$, as claimed.
\end{proof}
Consider the partition \cY\ in the above claim and observe that
$r_1(\cup_{Y\in \cY}D[Y]) + r_2(\cup_{Y\in \cY}\delta^{in}_D(Y))
=k(|V|-|\cY |) + \sum_{Y\in \cY} r_2(\delta^{in}_D(Y)) < k(|V|-1) $, thus $\sum_{Y\in \cY} r_2(\delta^{in}_D(Y)) < k(|\cY |-1)$. Let
$\cX=\{Y\in \cY: r_2(Y)<k \}$ and note that $\sum_{X\in \cX}
r_2(\delta^{in}_D(X)) < k(|\cX|-1)$ holds as well.  But
$r_2(\delta^{in}_{D}(X))=r^\oplus(\delta^{in}_{D}(X))$ for every $X\in \cX$, thus we get a contradiction
with \eqref{eq:subpart}.
\end{proof}

\fi

Let us fix some $s\in V$. From now on we are interested in
\indep\ \krarb s, and we assume $r_v(\delta^{in}_D(v))=k$ for every
$v\in V-s$.  Let
\begin{multline}\cB^s=\{I\subseteq \delta^{out}_D(s): |I|= k \text{ and} \\
\exists \text{ \indep\ \krarb[s] $F\subseteq A$ s.t.\ $I=F\cap
\delta^{out}_D(s)$}\}.\label{eq:Bs}
\end{multline}
 Our aim below is to show that $\cB^s$ is the
family of bases of a matroid on ground set $\delta_D^{out}(s)$.
For an arc set $I\subseteq \delta^{out}_D(s)$, we use the notation $I\cup D[V-s]$
for the digraph obtained
from $D$ by deleting the edges of $ \delta^{out}_D(s) - I$.
\ifsodaversion
The proof of Lemma  \ref{lem:rankin} can be found in the Appendix.
\fi

\begin{lemma}\label{lem:rankin}
Let $D=(V,A)$ be a digraph, let $s\in V$, and let
$M_v=(\delta^{in}_D(v), r_v)$ be matroids of rank $k$ for every $v\in V-s$. The following properties are
equivalent for $I\subseteq
\delta^{out}(s)$.
\begin{enumerate}[(i)]
\item $I\in \cB^s$, \label{it:1}
\item $|I|=k$ and $I$ satisfies
$r^\oplus(\delta^{in}_{I\cup D[V-s]}(X))\ge k$ for every non-empty $X\subseteq V-s$,
\label{it:2}
\item $|I|=k$ and $I$ satisfies
$|I\cap E| + r^\oplus(\delta^{in}_{D-E}(X))\ge k$ for every $E\subseteq \delta^{out}_D(s)$ and
non-empty $X\subseteq V-s$.
\label{it:3}
\end{enumerate}
\end{lemma}
\ifsodaversion

\else

\begin{proof}
It is clear that \eqref{it:1} implies \eqref{it:2}.
Let us prove that \eqref{it:2} implies \eqref{it:1}.  Let $D'=I\cup
D[V-s]$.
We will prove that there exists a \indep\  \karb\ in $D'$ by applying Theorem
\ref{thm:indepkarb}. Suppose that $\sum
\{r^\oplus(\delta^{in}_{D'}(X)): X\in \cX\} < k(|\cX|-1)$ for some
subpartition \cX. Note that we can assume
$r^\oplus(\delta^{in}_{D'}(X)) < k$ for every member $X$ of \cX, and
clearly $|\cX|>1$ has to hold. Therefore there must exist a member
$X\in\cX$ with $s\notin X$ and $r^\oplus(\delta^{in}_{D'}(X))<k$,
contradicting \eqref{it:2}.

Next we show that \eqref{it:1} implies \eqref{it:3}.  If $F\subseteq
A$ is a \indep\  \krarb[s] with $I=F\cap \delta^{out}_D(s)$,
$E\subseteq \delta^{out}_D(s)$, and $X\subseteq V-s$, then $k\le
\varrho_F(X) = \varrho_{F\cap E}(X)+ \varrho_{F- E}(X) \le |F\cap E| +
r^\oplus(\delta^{in}_{D-E}(X)) = |I\cap E| +
r^\oplus(\delta^{in}_{D-E}(X))$.  Finally, we show that \eqref{it:3}
implies \eqref{it:2}. Take some non-empty $X\subseteq V-s$, let
$E=(\delta_D^{out}(s)\cap \delta_D^{in}(X))-I$ and apply
the property in \eqref{it:3} for $X$ and $E$ to obtain \eqref{it:2}.
\end{proof}
\fi




Consider the following polyhedron.
\begin{align}
P=\{x\in \Rset^{\delta^{out}(s)}:\ & x\ge 0,\\
&x(E)\ge k-r^\oplus(\delta^{in}_{D-E}(X)) \mbox{ for every }\label{eq:xZ}
E\subseteq \delta^{out}_D(s) \mbox{ and }  \emptyset \neq X\subseteq V-s\}.
\end{align}
Clearly, $P$ is non-empty if and only if
$r^\oplus(\delta^{in}_{D}(X))\ge k$ for every non-empty $X\subseteq
V-s$ (the condition is necessary because
otherwise \eqref{eq:xZ} does not hold for $E=\emptyset$; on the other hand, if
this condition holds, then $k{\mathbf 1}\in
P$).
Furthermore, it is enough to require \eqref{eq:xZ} for non-empty
subsets $X$ that contain the head of every arc of $E$.
We can also observe that non-negativity of $x$  is implied by \eqref{eq:xZ} in the definition of $P$.
Indeed, let $st\in A$ be arbitrary and apply \eqref{eq:xZ} for $E=\{st\}$ and
$X=\{t\}$ to get $x(st)\ge k-r_t(\delta^{in}(t)-st)\ge 0$.

From now on we assume that $P$ is non-empty. Define the set function $p:2^{\delta^{out}_D(s)}\to \Rset$ as
\begin{equation}\label{eq:pZ}
p(E) = \max \{k-r^\oplus(\delta^{in}_{D-E}(X)): \emptyset \ne X\subseteq V-s\}.
\end{equation}
Note that $p\le k$ and
$p(\delta_D^{out}(s))=k-r^\oplus(\delta^{in}_{D-\delta_D^{out}(s)}(V-s))=k$. Furthermore,
$p(\emptyset)=0$ ($p(\emptyset)\le 0$ by the non-emptiness of $P$, and take any $v\in V-s$ and use $r_v(\delta_D^{in}(v))=k$ to obtain $p(\emptyset)\ge k-r^\oplus(\delta_D^{in}(v))=0$), and $p$ is monotone increasing.  With this definition, $P$ is described as

\[P=\{x\in \Rset^{\delta^{out}_D(s)}: x(E)\ge p(E)\mbox{ for every }E\subseteq \delta^{out}_D(s)\}.\]

Recall that a function $p:2^S\to \Rset$ is \textbf{near
  supermodular} if $p(X)+p(Y)\le p(X\cap Y) + p(X\cup Y)$ holds for
every intersecting pair $X,Y\subseteq V$ of non-separable sets, where
a set $X$ is separable if there exists a partition $X_1, X_2, \dots,
X_t$ of $X$ such that $p(X)\le \sum_i p(X_i)$.
\ifsodaversion
The proof of Theorem  \ref{thm:supermod} can be found in the Appendix.
\fi
\begin{thm}\label{thm:supermod}
The function $p$ defined in \eqref{eq:pZ} is near supermodular.
\end{thm}

\ifsodaversion

\else

For the proof of Theorem \ref{thm:supermod} we need the following claims.

\begin{claim}\label{cl:rplus}
Let $E_1, E_2\subseteq \delta_D^{out}(s)$ and  $X_1, X_2\in
V-s$ be arbitrary, then
\begin{equation}
r^\oplus(\delta^{in}_{D-E_1}(X_1))+ r^\oplus(\delta^{in}_{D-E_2}(X_2)) \ge r^\oplus(\delta^{in}_{D-(E_1\cup E_2)}(X_1\cup X_2)) + r^\oplus(\delta^{in}_{D-(E_1\cap E_2)}(X_1\cap X_2)).
\end{equation}
\end{claim}
\begin{proof}
By the properties of the direct sum, it is enough to show the following for an arbitrary $v\in V$, where $\Delta$ denotes $\delta^{in}_D(v)$.
\begin{equation}\label{eq:rvsub}
r_v(\delta^{in}_{\Delta-E_1}(X_1))+ r_v(\delta^{in}_{\Delta-E_2}(X_2)) \ge r_v(\delta^{in}_{\Delta-(E_1\cup E_2)}(X_1\cup X_2)) + r_v(\delta^{in}_{\Delta-(E_1\cap E_2)}(X_1\cap X_2)).
\end{equation}
If $v\notin X_1\cup X_2$, then there is nothing to prove, every term is
zero on both sides of \eqref{eq:rvsub}. If $v\in X_1-X_2$, then the
second term is zero on both sides of \eqref{eq:rvsub}, and the inequality
$r_v(\delta^{in}_{\Delta-E_1}(X_1))\ge
r_v(\delta^{in}_{\Delta-(E_1\cup E_2)}(X_1\cup X_2))$ is
implied by the mononicity of $r_v$. Clearly, the case $v\in X_2-X_1$
is analogous, therefore assume $v\in X_1\cap X_2$. Observe that \eqref{eq:rv1} and \eqref{eq:rv2} holds. For an illustration, see Figure \ref{fig:rv}.
\begin{eqnarray}\label{eq:rv1}
\delta^{in}_{\Delta-E_1}(X_1)\cap \delta^{in}_{\Delta-E_2}(X_2) = \delta^{in}_{\Delta-(E_1\cup E_2)}(X_1\cup X_2)\\
\label{eq:rv2} \delta^{in}_{\Delta-E_1}(X_1)\cup \delta^{in}_{\Delta-E_2}(X_2) = \delta^{in}_{\Delta-(E_1\cap E_2)}(X_1\cap X_2).
\end{eqnarray}
\begin{figure}
\begin{center}
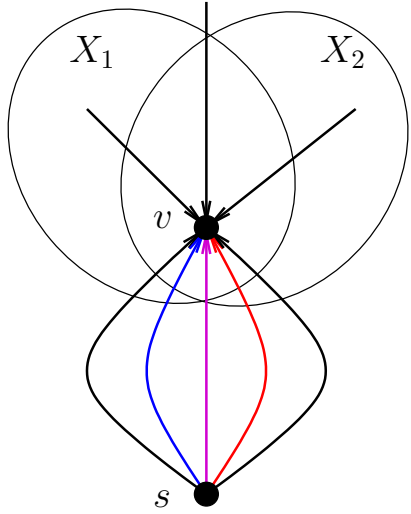
\caption{An illustration for proving \eqref{eq:rv1} and \eqref{eq:rv2}. The arcs of $(E_1-E_2)\cap \delta_D^{in}(v)$ are coloured blue, those in $(E_2-E_1)\cap \delta_D^{in}(v)$ are red, and those in $(E_1\cap E_2)\cap \delta_D^{in}(v)$ are magenta. That is, $\Delta-E_1$ is the set of arcs in the figure that are neither blue, nor magenta, etc.}
\label{fig:rv}
\end{center}
\end{figure}
This, together with the submodularity of $r_v$, finishes the proof.
\end{proof}


\newcommand{\dels}{\ensuremath{\delta_D^{out}(s)}}

Let us introduce the following notation. For a set $E\subseteq
\delta_D^{out}(s)$, let $X_E\subseteq V-s$ be an arbitrary subset that
attains the maximum in the definition \eqref{eq:pZ} of $p(E)$ (that
is, $X_E\ne \emptyset$ and $p(E)=k-r^\oplus(\delta^{in}_{D-E}(X_E))$).

\begin{claim}\label{cl:headin}
If $E\subseteq \dels$ is non-separable, then $X_E$ contains the head of every arc of $E$.
\end{claim}
\begin{proof}
Suppose not and let $E_1\subsetneq E$ be the subset of those arcs which
have their head in $X_E$. Then
$p(E)=k-r^\oplus(\delta^{in}_{D-E}(X_E))=k-r^\oplus(\delta^{in}_{D-E_1}(X_E))\le
p(E_1)$. But then
$p(E)\le p(E_1)+p(E-E_1)$ by the non-negativity of $p$, contradicting
the non-separability of $E$.
\end{proof}

\begin{proof}[Proof of Theorem \ref{thm:supermod}]
Let $E_1, E_2\subseteq\delta^{out}_D(s) $ be non-separable sets so
that $E_1\cap E_2\ne \emptyset$. By Claim \ref{cl:headin}, $X_i=X_{E_i}$ contains the head of each arc of $E_i$ for both $i=1,2$.
This implies that $X_1\cap X_2\ne \emptyset$, and  Claim \ref{cl:rplus} gives
\begin{eqnarray*}
p(E_1)+p(E_2)= \sum_{i=1,2}k-r^\oplus(\delta^{in}_{D-E_i}(X_i))\le \\
2k-\left(r^\oplus(\delta^{in}_{D-(E_1\cup E_2)}(X_1\cup X_2)) + r^\oplus(\delta^{in}_{D-(E_1\cap E_2)}(X_1\cap X_2))\right)\le \\
p(E_1\cap E_2)+p(E_1\cup E_2).
\end{eqnarray*}
\end{proof}

\fi

Theorems \ref{thm:wedge} and \ref{thm:supermod} imply that $P$ is an integer polyhedron. It is also easy to see the following.

\begin{cor}
The polyhedron $B=\{x\in P: x( \delta^{out}_D(s)) = k\}$ (if not empty) is a base
polyhedron of a matroid. It is the convex hull of incidence vectors of members of
$\cB^s$.
\end{cor}
\begin{proof}
We show that $x\in B$ implies $x\le 1$. This, together with Theorems
\ref{thm:wedge} and \ref{thm:supermod} and Lemma \ref{lem:rankin}, proves the
corollary. Take $x\in B$ and $st\in A$. Let $E=\delta^{out}(s)-st$ and
$X= V-s$. By \eqref{eq:xZ}, we have $k-x(st)=x(E)\ge k-r^\oplus(\delta^{in}_{st}(V-s)) = k-r_t(\{st\}) \ge k-1$.
\end{proof}

The following claim describes the (fully supermodular) truncation of $p$.
\ifsodaversion
The proof can be found in the Appendix.
\fi

\begin{claim}\label{cl:wedge}
For any $E\subseteq \delta^{out}_D(s)$,
\begin{equation}\label{eq:pwedge2}
p^\wedge(E)=\max\left\{\sum_{X \in \cX} (k-r^\oplus(\delta^{in}_{D-E}(X))): \cX \mbox{ is a subpartition of }V-s\right\}.
\end{equation}
\end{claim}
\ifsodaversion
\else
\begin{proof}
Let $E\subseteq \delta^{out}_D(s)$ and let \cH\ be a partition of
$E$ that gives $p^\wedge(E)=\sum \{p(H): H\in \cH\}$ and, subject to
this, $|\cH|$ is minimal. Clearly, every $H\in \cH$ is non-separable.
We claim that $\{X_H: H\in \cH\}$ is a subpartition of $V-s$. If there
exist $H_1, H_2\in \cH$ so that $X_{H_1}\cap X_{H_2}\ne \emptyset$,
then (by Claim \ref{cl:rplus}) $p(H_1)+p(H_2)\le 2k-(r^\oplus(\delta^{in}_{D-(H_1\cap
  H_2)}(X_{H_1}\cap X_{H_2})) + r^\oplus(\delta^{in}_{D-(H_1\cup
  H_2)}(X_{H_1}\cup X_{H_2})))\le p(H_1\cup H_2) + p(\emptyset) =
p(H_1\cup H_2)$, therefore $\cH'=\cH-\{H_1, H_2\} + \{H_1\cup H_2\}$
also gives $p^\wedge(E)=\sum \{p(H): H\in \cH'\}$, contradicting our
choice of \cH.
\end{proof}

\fi

\begin{cor}\label{cor:indepmatroid}
Let $D=(V,A)$ be a digraph, let $s\in V$, and let $M_v=(\delta^{in}_D(v), r_v)$
($v\in V-s$) be matroids of rank $k$.  The family
$\cB^s$ defined in \eqref{eq:Bs}, if non-empty, defines the family of bases of a
matroid $M^s$ on ground set $\delta^{out}_D(s)$. The family is not empty if
and only if
\begin{enumerate}[(a)]
\item \label{it:Pnempty} $r^\oplus(\delta^{in}_{D}(X)\ge k$ for every non-empty $X\subseteq
V-s$, and
\item \label{it:Bnempty} $\sum \{k-r^\oplus(\delta^{in}_{D[V-s]}(X)): X\in \cX\}\le k $ for every  subpartition \cX\ of $V-s$.
\end{enumerate}
The rank function of $M^s$ is given by the following formula for any $E\subseteq \delta_D^{out}(s)$:
\[r^s(E)=\min\left\{\sum_{X\in \cX} r^\oplus(\delta^{in}_{E\cup D[V-s]}(X)) - k(|\cX|-1): \cX \mbox{ is a subpartition of }V-s\right\}.\]
\end{cor}
\ifsodaversion
The proof of Corollary \ref{cor:indepmatroid} is in the Appendix.
\else
\begin{proof}
Consider the function $p^\wedge$ defined by \eqref{eq:pwedge2}.
By Theorem \ref{thm:wedge}, $p^\wedge$ is monotone increasing and
supermodular, and $P=\{x\in \Rset^{\delta^{out}(s)}:
x(E)\ge p^\wedge(E)\mbox{ for every }E\subseteq
\delta^{out}_D(s)\}$ if $P$ is non-empty. Thus $B=\{x\in P: x( \delta^{out}_D(s)) = k\}$
is not empty if and only if $P\ne \emptyset$ and
$p^\wedge(\delta^{out}_D(s))=k$, that is, if and
only if both \eqref{it:Pnempty} and \eqref{it:Bnempty} hold.
Since the fully supermodular function describing the base polyhedron $B$ is $p^\wedge$, it is the co-rank function of the matroid $M^s$, and its rank function is given by the formula
\begin{multline*}
 r^s(E)=p^\wedge(\delta^{out}_D(s))- p^\wedge(\delta^{out}_D(s)-E)=k-p^\wedge(\delta^{out}_D(s)-E) \\= \min\{\sum_{X\in \cX} r^\oplus(\delta^{in}_{E\cup D[V-s]}(X)) - k(|\cX|-1): \cX \mbox{ is a subpartition of }V-s\}.
 \end{multline*}

\end{proof}
\fi

\section{Matroidal description of \cL-tight \karb s} \label{sec:Ltight}

Let $D=(V,A)$ be a digraph, let $\cL\subseteq 2^V$ be a laminar family, and
assume that there exists an \cL-tight \karb\ in
$D$. Without loss of generality, we also assume that $V$ and all singletons
are in \cL. Let furthermore $D^+$ denote the $(|A|+k)$-extension of $D$.
The \cL-tight \karb s in $D^+$ are all rooted at $s$ and, since $V\in \cL$, there is a natural
(though not one-to-one) correspondence between
\cL-tight \karb s in $D$ and those in
$D^+$.
For $W\in \cL$, let $D_W$ denote the digraph obtained from $D^+$ by contracting
$V+s-W$ to a single node $s_W$ and removing the loops that arise.
Note that there is a natural bijection between $\delta^{out}_{D_W}(s_W)$ and $ \delta^{in}_{D^+}(W)$; we will basically identify these two arc-sets in the discussion below.
The main theorem of this section is the following.
\ifsodaversion
The proof of Theorem   \ref{thm:MF} can be found in the Appendix.
\fi

\begin{thm}\label{thm:MF}
The family $\cB_W=\{I\subseteq \delta^{out}_{D_W}(s_W): |I|= k$ and $I$
can be extended to an $\cL[W]$-tight \krarb[s_W] in
$D_W\}$ forms the family of bases of a matroid
$M_W=(\delta^{out}_{D_W}(s_W), r_W)$.
\end{thm}
\newcommand{\Fc}{\hat{F}}
\newcommand{\Dc}{\hat{D}}
\newcommand{\Wc}{\hat{W}}
\newcommand{\Bc}{\hat{B}}
\newcommand{\Ec}{\hat{E}}

\ifsodaversion
\else

\begin{proof}
We recursively show that the family $\cB_W$ indeed defines a matroid $M_W$ for every $W\in \cL$. For
the singletons $\{v\}\in \cL$ it is clear that $M_{\{v\}}$ is the
uniform matroid of rank $k$ on ground set $\delta_{D^+}^{in}(v)$. Let $W\in
\cL$ be a non-singleton, and assume that $M_{W'}$ has already been
defined for every $W'\in \cL$ that is a proper subset of $W$. Let $W_1, W_2,
\dots , W_l$ be the maximal members of $\cL[W]-W$, and let us contract each
$W_i$ into a single node $w_i$ ($i=1,2,\dots,l$).  Let
 $\Wc = W/\{W_1, W_2, \dots , W_l\}$ be the set obtained
from $W$ by these contractions, and
similarly, for a subgraph $(W+s_W, E)$ of $D_W$ we use the
notation $\Ec = E/\{W_1, W_2, \dots , W_l\}$ to mean the graph
obtained from $(W+s_W, E)$ by the contractions (and deletion of the loops that arise).
In particular, let $\Dc=D_{W}/\{W_1, W_2,
\dots , W_l\}$. The matroids $M_{W_i}$ naturally give rise to matroids
$M_{w_i}=(\delta^{in}_{\Dc}(w_i), r_{w_i})$ for
every $i$; let $\cM=\{M_{w_1}, \dots, M_{w_l}\}$.
\begin{claim}
If $F\subseteq A(D_{W})$ is an $\cL[W]$-tight
\krarb[s_W], then $\Fc$ is
\cM-\indep. Conversely, if
$F'\subseteq \Dc$ is an \cM-\indep\  \krarb[s_W] in $\Dc$ and
$|\delta_{F'}^{out}(s_W)|=k$, then there exists an $\cL[W]$-tight
\krarb[s_W] $F\subseteq A(D_{W})$ such that $\Fc=F'$.
\end{claim}
\begin{proof}
The first statement is clear from the definition of the matroids
$M_{w_i}$. For the other direction, let $F'\subseteq \Dc$ be an
\cM-\indep\ \krarb[s_W] in $\Dc$, such that $|\delta_{F'}^{out}(s_W)|=k$.
Consider $F'$ as a
subgraph of $D_W$, and note that $\delta^{in}_{F'}(W_i)$ is a base of
$M_{W_i}$ for every $i$. By the definition of $M_{W_i}$,
$\delta^{in}_{F'}(W_i)$ can be extended to an $\cL[W_i]$-tight
arborescence $F_i$ in $D_{W_i}$ for every $i$. The \krarb[s_W] $F=F'\bigcup \cup_iF_i$ is $\cL[W]$-tight
and $\Fc= F'$, as required.
\end{proof}

The claim implies that $\cB_W$ consists of the arc sets of size $k$ that can be obtained as the arcs incident to $s_W$
of an \cM-\indep\ \krarb[s_W], so the statement of the theorem
follows from Corollary \ref{cor:indepmatroid}.
\end{proof}
\fi

\begin{cor}\label{cor:matr}
The matroids defined in Theorem \ref{thm:MF} have the property that
a $k$-arborescence $F\subseteq A(D^+)$ is \cL-tight if and only if $F\cap \delta^{in}_{D^+}(W)$ is a base of $M_W$ for every $W\in \cL$.
\end{cor}

A recursive formula for the rank function $r_W$ of the matroid $M_W$ defined in Theorem \ref{thm:MF}
can be deduced from Corollary \ref{cor:indepmatroid}. We state this recursive formula expicitly below because it
will be used extensively. Let $W_1, \dots, W_l$ denote
the maximal members of $\cL[W]-W$. For an arc set $E\subseteq
\bigcup_{i=1}^l\delta^{in}_{D^+}(W_i)$, we use the notation
$r^\oplus_{W}(E)=\sum_{i=1}^l r_{W_i}(E\cap \delta^{in}_{D^+}(W_i))$.
A subset $X$ of $W$ is called $\cL[W]$-\textbf{compatible} if it is the union of some maximal members of $\cL[W]-W$.
A subpartition \cP\ of $W$ is $\cL[W]$-\textbf{compatible} if every member of \cP\ is $\cL[W]$-compatible.

\begin{cor} \label{cor:recursive}
Let $W\in \cL$ and $E\subseteq \delta^{in}_{D^+}(W)$. If $|W|=1$, then $r_W(E)=\min\{k, |E|\}$; otherwise
  \[r_W(E)= \min\{\sum_{X\in \cX} r_{W}^\oplus(\delta^{in}_{E\cup D[W]}(X)) - k(|\cX|-1): \cX \mbox{ is an $\cL[W]$-compatible subpartition of }W\}.\]
\end{cor}

Theorem \ref{thm:MF} for $W=V$ gives the following corollary.
\begin{cor}
The convex hull of root vectors of \cL-tight \karb{s} in $D$ is a base polyhedron.
\end{cor}


Theorem \ref{thm:MF} in itself does not imply that the root vectors of minimum-cost $k$-arborescences
also determine a base polyhedron, because we have to deal with mandatory arcs,
i.e.\ the arcs of $A_1$ in Corollary \ref{cor:slackness}.
\ifsodaversion
However, this problem can be handled and we can prove the following theorem (see the details in the Appendix).
\else
The following transformation solves this issue.

\newcommand{\mandconstr}{{\sc mandatory arc transformation}}
\begin{figure}
\begin{center}
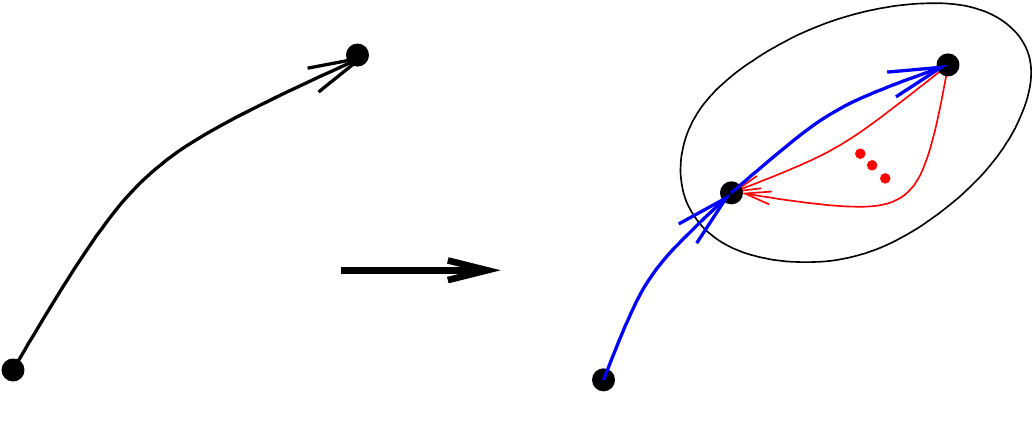
\caption{An illustration for the \mandconstr.}
\label{fig:mand}
\end{center}
\end{figure}

\paragraph{\mandconstr} Given a digraph $D=(V,A)$, a node $s\in V$, an
arc $a=uv$ (where $u,v\in V-s$), and a laminar family $\cL\subseteq
2^{V-s}$, we construct a digraph $D'=(V+x_a, A-a+B_a)$, where
$B_a=\{ux_a, x_av\}\cup\{k-1$ parallel copies of $vx_a\}$.  Let
furthermore $\cL'\subseteq 2^{V+x_a}$ be defined as $\cL'=\{W\in \cL:
v\not\in W\}\cup\{W+x_a: v\in W\in \cL\}$ (note that $\{v\}\in \cL$ implies that $\{x_a,v\}\in \cL'$).
See Figure \ref{fig:mand} for an illustration. It is
easy to check that $\cL'$ is laminar.

\begin{claim}\label{cl:phi}
For an \cL-tight \krarb\ $F\subseteq A$ containing $a=uv$, let
$\phi(F) = F-a+B_a$. Then $\phi$ is a bijection between \cL-tight
\krarb{s} containing $a$ in $D$ and $\cL'$-tight \krarb{s} in $D'$.
\end{claim}
\begin{proof}
First we show that if $F\subseteq A$ is an \cL-tight \krarb\  containing $a$, then $\phi(F)$
is an $\cL'$-tight \krarb\ in $D'$. Let $F_1, F_2,
\dots, F_k$ be a decomposition of $F$ into $k$ $s$-rooted
arborescences and assume that $a\in F_1$. Let
$F_1'=F_1-a+\{ux_a,x_av\}$, and let $F_i'=F_i$ plus a copy of the arc
$vx_a$ for every $i=2,\dots, k$. Then $F_1', F_2', \dots, F_k'$ is a
decomposition of $\phi(F)$ into $k$ $s$-rooted arborescences in $D'$,
so $\phi(F)$ is indeed an \krarb\ in $D'$. Furthermore, $\phi(F)$
is $\cL'$-tight, as $\phi(F)[\{x_a,v\}]$ is a \karb, and the indegree
of any other set $W\in \cL'$ in the subgraph $\phi(F)$ is $k$.

For the other direction, let $F'\subseteq A'$ be an arbitrary
$\cL'$-tight \krarb\ in $D'$. Since $F'[\{x_a,v\}]$ is a \karb\ and
$\varrho_{F'}(x_a)=k$, $B_a\subseteq F'$ must hold. Let
$F=F'-B_a+a$; we show that $F$ is a \cL-tight \krarb\ in $D$ --
since $a\in F$ and $F'=\phi(F)$, this completes the proof. Let $F_1',
F_2', \dots, F_k'$ be a decomposition of $F'$ into $k$ $s$-rooted
arborescences in $D'$, and assume that $ux_a\in F_1'$. Then clearly
$x_av$ is in $F_1'$ too, so $F_1'-\{ux_a,x_av\}+a, F_2'-vx_a, F_3'-vx_a,
\dots, F_k'-vx_a$ is a decomposition of $F$ into $k$ $s$-rooted
arborescences in $D$. The \cL-tightness of $F$ can be shown similarly.
\end{proof}

Using this transformation we can now prove the following.
\fi

\begin{thm}\label{thm:optimalroot}
The convex hull of the root vectors of optimal $k$-arborescences is a
base polyhedron.
\end{thm}
\ifsodaversion
\else
\begin{proof}
Given a digraph $D=(V, A)$ and a cost function $c:A\to \Rset$, let $\alpha=k+1$,
$\beta=\sum_{a\in A}c(a)+1$, and let $(D^+,c^+)$ be the $(\alpha,\beta)$-extension of $(D,c)$.
By previous remarks, optimal \karb{s} in $D$ and optimal \karb{s} in $D^+$
correspond to each other in a natural way (and \karb{s} in $D^+$ are
rooted at $s$).  By Corollary \ref{cor:slackness}, there exists a laminar family
$\cL\subseteq 2^V$ and two disjoint sets $A_0, A_1\subseteq A^+$, such
that a \karb\ $F\subseteq A^+$ is
optimal if and only if $A_1\subseteq F\subseteq A^+-A_0$ and $F$ is
$\cL$-tight. Due to symmetry, $A_1$ contains either all or none of
the parallel arcs between $s$ and a given node $v\in V$. Since there are $k+1$
parallel arcs, the former is impossible, so $A_1\subseteq A$.

Starting with $D^+-A_0$, repeat the \mandconstr\  above for
every $a\in A_1$, to obtain $D'=(V+s+\{x_a:a\in A_1\}, A^+-(A_0\cup
A_1)+\cup_{a\in A_1}B_a))$ and the laminar family $\cL'\subseteq
2^{V+\{x_a:a\in A_1\}}$. For any $\cL$-tight \krarb\ $F\subseteq A^+$
with $A_1\subseteq F\subseteq A^+-A_0$, let $\phi(F)=F-A_1+\cup_{a\in
  A_1}B_a$. By Claim \ref{cl:phi}, $\phi$ defines a bijection between
$\cL$-tight \krarb{s} in $D^+-A_0$ containing $A_1$ and $\cL'$-tight
\krarb{s} in $D'$. By Corollary \ref{cor:matr}, the family
$\{I\subseteq \delta^{out}_{D'}(s): |I|=k$ and $I$ is contained in a   $\cL'$-tight
\krarb\ of $D'\}$ is the family of bases of a matroid. This implies that the convex hull of root vectors of
optimal $k$-arborescences in $D$ is a base polyhedron.
\end{proof}
\fi

\section{Blocking \cL-tight \karb{s}}\label{sec:blocking}

In this section we show that if $k$ is fixed, then there is a polynomial-time algorithm that finds a minimum transversal
of the family of $\cL$-tight \karb s.
Let $D=(V,A)$ be a digraph and let $\cL\subseteq 2^V$ be a laminar family.
We assume that $\cL$ contains $V$ and all the singletons, and that $D$ contains an \cL-tight \karb . Let
$D^+$ be the $\alpha$-extension of $D$, where $\alpha=|A|+k$. The minimum transversals
for $D$ and $D^+$ are the same because the arcs $sv$ have $|A|+k$
copies each, so these arcs never appear in a minimum transversal.
Recall that for $W\in \cL$, the digraph $D_W$ is obtained by contracting $V+s-W$ in $D^+$ to a single root node $s_W$.

In what follows, we will often use the matroids $M_W=(\delta^{in}_{D^+}(W),
r_W)$ for $W\in \cL$, as defined in Theorem \ref{thm:MF}. Furthermore,
we will often remove some subset of arcs $H\subseteq A$ from $D^+$
and we will usually denote $D^+-H$ by $D'$. Thus $D'_W$ for some $W\in \cL$ will denote the digraph obtained
from $D^+-H$ by contracting $V+s-W$ into a single node $s_W$. If $D'_W$ contains an
$\cL[W]$-tight \krarb[s_W] for some $W\in \cL$, then we can consider the modified matroid obtained by using $D'_W$ in place of $D_W$ in
Theorem \ref{thm:MF}. To emphasize the dependence of this matroid on $D'$, we denote it
by $M_{D',W}$, and its rank function by $r_{D', W}$. Likewise, we use the notation $r_{D', W}^\oplus(E)$ in place of
$r_{W}^\oplus(E)$ if we refer to the direct sum defined using $D'$.

For a non-singleton $W\in \cL$ and an arc set $E \subseteq \delta^{in}_{D^+}(W)$, we say that
an $\cL[W]$-compatible subpartition \cX\ of $W$ \textbf{determines} $r_W(E)$ if
$r_W(E)= \sum_{X\in \cX} r_{W}^\oplus(\delta^{in}_{E\cup D[W]}(X)) - k(|\cX|-1)$. By
Corollary \ref{cor:recursive}, such a subpartition exists.
Notice that if $\cX$ determines $r_W(E)$, then $r_{W}^\oplus(\delta^{in}_{E\cup D[W]}(X))\le k$ for every $X\in \cX$.
Moreover, if $r_{W}^\oplus(\delta^{in}_{E\cup D[W]}(X))=k$ for some $X\in \cX$, then $\cX-X$ also determines $r_W(E)$.
In particular, if $r_W(E)=k$, then $r_W(E)$ is determined by the empty subpartition.

Our first lemma shows that  the rank of an arc set cannot decrease by more than one if we remove only one arc from $D$.
\ifsodaversion
The proof of Lemma \ref{lem:rank1} can be found in the Appendix.
\fi

\begin{lemma}\label{lem:rank1}
Let $E \subseteq \delta^{in}_{D^+}(W)$, and let $D'=D^+-e$ for an arbitrary arc $e \in D_W$ (not necessarily in $E$).
If $D'_W$ contains an $\cL[W]$-tight \krarb[s_W], then
\[r_W(E)-1 \leq r_{D',W}(E-e)\leq r_W(E).\]
\end{lemma}
\ifsodaversion
\else
\begin{proof}
Let $E'=E-e$. The inequalities $r_{D',W}(E')\leq r_W(E')\leq r_W(E)$ follow from the definition of the rank. We prove the remaining
inequality by induction on the size of $\cL[W]$; it is clearly true if $W$ is a singleton.
Otherwise, by Corollary \ref{cor:recursive}, there is an $\cL[W]$-compatible subpartition $\cX$ of $W$ that determines $r_{D',W}(E')$, i.e.\
$r_{D',W}(E')=\sum_{X \in \cX} r_{D', W}^\oplus (\delta^{in}_{E' \cup D'[W]}(X))- k(|\cX|-1)$.
We know by induction that
$r_{D', W}^\oplus (\delta^{in}_{E' \cup D'[W]}(X))\geq r_W^\oplus (\delta^{in}_{E \cup D[W]}(X))-1$ for every $X \in \cX$,
and the ranks are different for at most one member of $\cX$,
since $e\in D_{W_i}$ for at most one $W_i$. This proves the inequality because
$r_W(E) \leq \sum_{X \in \cX} r_W^\oplus (\delta^{in}_{E \cup D[W]}(X))- k(|\cX|-1)$.
\end{proof}
\fi

The next result is a characterization of inclusionwise minimal transversals lying inside $A$.
\ifsodaversion
The proof of Theorem \ref{thm:minarcset} can be found in the Appendix.
\fi

\begin{thm}\label{thm:minarcset}
 Let $H \subseteq A$ be an inclusionwise minimal transversal of the family of $\cL$-tight \karb s in $D^+$.
 Let $D'=D^+-H$ and  let $W \in \cL$ be an inclusionwise minimal member of $\cL$ for which $D'_{W}$
 does not contain an $\cL[W]$-tight \krarb[s_{W}].
Then $H\subseteq D[W]$, and there is an $\cL[W]$-compatible subpartition $\cX$ of $W$ such that
$\sum_{X \in \cX} r_{D', W}^\oplus (\delta^{in}_{D'[W]}(X))= k(|\cX|-1)-1$.
\end{thm}

\ifsodaversion
\else

\begin{proof}
First note that $|W|>1$, since $H\subseteq A$. As $H \cap D_W$ is a transversal of $\cL[W]$-tight \krarb[s_{W}]s
(and hence of $\cL$-tight \karb s), minimality of $H$ implies that $H\subseteq D_W$.
Let $W_1, \dots, W_l$ be the maximal members of $\cL[W]-W$. By the choice of $W$, $D'_{W_i}$ contains an
$\cL[W_i]$-tight \krarb[s_{W_i}] for every $i$, thus $r_{D', W}^\oplus$ is well-defined.

Since $D'_{W}$ does not contain an $\cL[W]$-tight \krarb[s_{W}],
\eqref{it:Pnempty} or \eqref{it:Bnempty} fails to hold in Corollary \ref{cor:indepmatroid} for $r_{D', W}^\oplus$.
Suppose that $r_{D', W}^\oplus(\delta^{in}_{D'_W}(X))<k$ for some $\cL[W]$-compatible subset
$X \subseteq W$. Then there is a set $W_i$ such that
$r_{D',W_i}(\delta^{in}_{D'}(W)\cap \delta^{in}_{D'}(W_i))<k$. However, since we did not delete any arc leaving $s$, and already the arcs going from $s$ to $W_i$ have rank $k$ in $M_{W_i}$, we get (by monotonicity of $r_{W_i}$) that $r_{W_i}(\delta^{in}_{D'}(W)\cap \delta^{in}_{D'}(W_i))=k$, a contradiction.


Thus \eqref{it:Bnempty} fails to hold in Corollary \ref{cor:indepmatroid}, that is,
$\sum_{X \in \cX} r_{D', W}^\oplus (\delta^{in}_{D'[W]}(X))< k(|\cX|-1)$ for some
$\cL[W]$-compatible subpartition $\cX$ of $W$. As $H$ is inclusionwise minimal and the removal of an arc can
decrease a rank by at most one according to Lemma \ref{lem:rank1}, the left hand side must be equal to
$k(|\cX|-1)-1$. Since the formula involves only arcs in $D'[W]$, minimality also implies that $H \subseteq D[W]$.
\end{proof}
\fi


The characterization in the theorem does not lead automatically to an efficient algorithm for finding a transversal
of minimum size. In fact, for a given $X$ with $r_{W}^\oplus (\delta^{in}_{D[W]}(X))=k$, it is not clear how to compute
the minimmum number of arcs that have to be removed in order to decrease the rank by one.
However, the following lemma implies that if the rank is strictly smaller than $k$, then we can decrease it
by removing only one arc.
\ifsodaversion
The proof of Lemma \ref{lem:decrease_rank} can be found in the Appendix.
\fi

\begin{lemma}\label{lem:decrease_rank}
 Let $W \in \cL$ and $E\subseteq \delta^{in}_{D^+}(W)$ such that $0<r_W(E)<k$. Then there exists an arc
 $e\in E\cup D[W]$  such that either $D_W-e$ does not contain an $\cL[W]$-tight \krarb[s_{W}],
 or $r_{D',W}(E-e)=r_W(E)-1$, where $D' = D^+-e$.
\end{lemma}

\ifsodaversion
\else

\begin{proof}
The proof is by induction on $|W|$; the claim is clearly true if $W$ is a singleton.
Let $W_1, \dots, W_l$ be the maximal members of $\cL[W]-W$.
By Corollary \ref{cor:recursive}, there exists an $\cL[W]$-compatible subpartition \cX\ of $W$ that determines $r_W(E)$.
We can choose a set $X\in \cX$ and an index $i$  for which $W_i\subseteq X$ and
\[0<r_{W_i}(\delta^{in}_{E \cup D[W]}(X)\cap \delta^{in}_{E \cup D[W]}(W_i))<k.\]
Let $\Delta$ denote $\delta^{in}_{E \cup D[W]}(X)\cap \delta^{in}_{E \cup D[W]}(W_i)$. By induction, there is an arc
$e\in \Delta \cup D[W_i]$ such that either $D'_{W_i}$ does not contain an $\cL[W_i]$-tight \krarb[s_{W_i}]
(where $D'$ is the digraph obtained by removing $e$),
or $r_{D',W_i}(\Delta-e)=r_{W_i}(\Delta)-1$.
The latter possibility means that $r_{D',W}(E-e)< r_W(E)$; on the other hand, the rank can decrease by at
most one by Lemma \ref{lem:rank1}.
\end{proof}
\fi

We can formulate a similar statement for an $\cL[W]$-compatible subset
of $W$, which easily follows from the previous lemma.

\begin{lemma}\label{lem:decrease_rankX}
 Let $W \in \cL$, let $X\subseteq W$ be an $\cL[W]$-compatible set, and let $E\subseteq \delta^{in}_{D^+}(X)$ such that $0<r^\oplus_W(E)<k$.
 Then there exists an arc
 $e\in E\cup D[X]$  such that either $D_W-e$ does not contain an $\cL[W]$-tight \krarb[s_{W}],
 or $r^\oplus_{D',W}(E-e)=r^\oplus_W(E)-1$, where $D' = D^+-e$.\qed
\end{lemma}

Let $\gamma$ be the minimum size of a transversal of the family of $\cL$-tight \karb s. Using the above lemma, we will show that
if $\gamma \geq k$, then there exists a minimum transversal having a special structure. This will lead to a polynomial
algorithm for fixed $k$ the following way: first we check every arc subset of size at most $k-1$; if none of these is a
transversal, then we look for a minimum transversal among those having the special structure. As we will see, this can be done in polynomial time
using the results in \cite{mincostarb}.

We start with an easy corollary of Lemma \ref{lem:decrease_rankX} that describes a case that cannot happen when $\gamma \geq k$; the proof is
left to the reader.

\begin{cor}\label{cor:gamma<k}
If there exists $W\in \cL$ and two nonempty disjoint $\cL[W]$-compatible sets $X_1, X_2\subseteq W$
with $r^\oplus_W(\delta^{in}_{D[W]}(X_j))<k$ for both $j=1,2$, then $\gamma < k$. \qed
\end{cor}



To describe the special structure of the minimum transversal that we are looking for, we use a set function that also played a
crucial role in the $k=1$ case that was solved in \cite{mincostarb}.
For  $W\in \cL$ and $Z \subseteq W$, we define
\[f_W(Z) := |\{e \in D[W]: e \in \delta^{in}(Z),\ e \notin \delta^{out}(W') \text{ if $W' \in \cL[W]$ and $W' \cap Z \neq \emptyset$ }\}|.\label{eq:f_F}\]
If $D'$ is a digraph different from $D$, then we use $f_{D', W}(Z)$ to denote the analogous set function for $D'$.
The following claim was proved for $k=1$ in \cite[Lemma 3]{mincostarb}.

\begin{claim}\label{cl:Z_1Z_2}
Let $D=(V, A)$ be a digraph and $\cL\subseteq 2^V$ a laminar family. If
there exists an \cL-tight \karb\ in $D$, then $f_W(Z_1)+f_W(Z_2)\ge k$
for any $W\in \cL$ and nonempty disjoint sets $Z_1, Z_2\subseteq W$.
\end{claim}


\begin{proof}
Suppose for contradiction that there exists an \cL-tight \karb\ in $D$
and there exist
 $W\in \cL$ and nonempty disjoint sets  $Z_1, Z_2\subseteq W$ such that  $f_W(Z_1)+f_W(Z_2)\le k-1$.
Consider the digraph $D'$ obtained from $D$ the following way: for every arc $e \in \delta^{in}_{D[W]}(Z_j)$ for which
there exists $W' \in \cL[W]$ such that $W' \cap Z_j \neq \emptyset$ and $e \in \delta^{out}_{D[W]}(W')$, we change the
tail of $e$ to an arbitrary node in $W' \cap Z_j$ ($j=1,2$). This is the \textbf{tail-relocation operation} introduced in \cite{mincostarb}.
The following can be seen easily:
\begin{itemize}
\item If $F$ is an $\cL[W]$-tight \karb\ in $D$, then the corresponding arc set in $D'$ is also an $\cL[W]$-tight \karb;
\item $f_W(Z_j)=f_{D',W}(Z_j)=\varrho_{D'[W]}(Z_j)$ ($j=1,2$).
\end{itemize}
This contradicts $f_W(Z_1)+f_W(Z_2)\le k-1$, because the
existence of an $\cL[W]$-tight \karb\ implies $\varrho_{D'[W]}(Z_1)+\varrho_{D'[W]}(Z_2) \geq k$.
\end{proof}

Note that in the case $k=1$, \cite[Lemmas 3, 4]{mincostarb} state that
there exists an \cL-tight arborescence in $D$ if and only if
$f_W(Z_1)+f_W(Z_2)\ge 1$ for any $W\in \cL$ and nonempty disjoint sets
$Z_1, Z_2\subseteq W$.  Unfortunately, the analogous statement is not true
for $k>1$, as illustrated in Figure \ref{fig:Tamaspelda}.

\begin{figure}
\begin{center}
\scalebox{0.8}{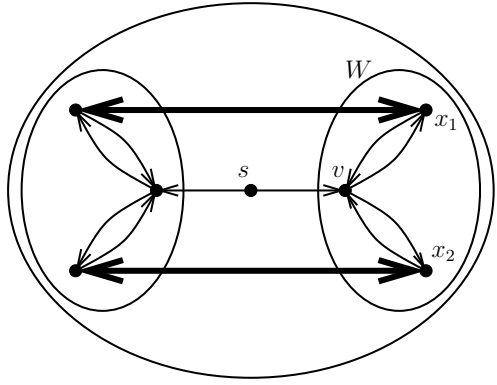}
\caption{A digraph that does not admit an \cL-tight $2$-arborescence.
Bold arcs are bidirected and have multiplicity 2, and \cL\ has 3 members, indicated by ellipses.
There is no \cL-tight $2$-arborescence, although  $\sum_{X\in \cX} f_{W}(X) \ge k(|\cX|-1)$ holds for every $W\in \cL$
and every $\cL[W]$-compatible subpartition \cX\ of $W$. Note that the arc $sv$ is a loop in the matroid $M_W$, and $r_W(\{sv\})$ is determined by the subpartition $\{\{x_1\}, \{x_2\}\}$.}
\label{fig:Tamaspelda}
\end{center}
\end{figure}

The following upper bound on the rank can be proved similarly to Claim \ref{cl:Z_1Z_2}.
\ifsodaversion
The proof can be found in the Appendix.
\else
\fi

\begin{lemma}\label{lem:ranklefW}
If $W\in \cL$ and $E\subseteq \delta^{in}_{D^+}(W)$, then $r_W(E)\le f_W(Z)+\varrho_E(Z)$ for every non-empty $Z\subseteq W$.
\end{lemma}
\ifsodaversion
\else
\begin{proof}
By the definition of the rank, there is an $\cL[W]$-tight \krarb[s_W] $F$ such that $|F\cap E|= r_W(E)$.
We apply the tail-relocation operation
described in the proof of Claim \ref{cl:Z_1Z_2}; let $D'$ be the modified digraph, and let $F'$ be the $\cL[W]$-tight \krarb[s_W] obtained from $F$.
On one hand, $f_W(Z)=f_{D',W}(Z)=\varrho_{D'[W]}(Z)$. On the other hand,
\[k \leq \varrho_{F'}(Z) \leq \varrho_{D'[W]}(Z)+\varrho_{E\cap F}(Z)+ \varrho_{F-E}(W)\leq \varrho_{D'[W]}(Z)+\varrho_{E}(Z)+(k-r_W(E)),\]
so $r_W(E) \leq \varrho_{D'[W]}(Z)+ \varrho_{E}(Z)=f_W(Z)+\varrho_E(Z)$, as required.
\end{proof}
\fi

\newcommand{\elem}{elementary}

Our next observation is that for some special arc sets the above formula is tight.
To describe these special arc sets, we use a recursive definition.
For $W\in \cL$ and $E\subseteq \delta^{in}_{D^+}(W)$, we say that $E$ is \textbf{$W$-\elem} if $r_W(E)<k$ and
\begin{itemize}
\item either $|W|=1$
\item or there exists an $\cL[W]$-compatible set $X \subseteq W$ such that the subpartition $\{X\}$ determines $r_W(E)$,
and $\delta^{in}_{E\cup D[W]}(X)\cap
  \delta^{in}_{E\cup D[W]}(W')$ is $W'$-\elem\ for every maximal member $W'$
  of $\cL[W]-W$.
\end{itemize}
Intuitively, an arc set is \elem\ if only subpartitions of cardinality 1 occur in its recursive rank formula.
Note that $E=\emptyset$ is $W$-elementary for every $W$, since $\{W\}$ determines $r_W(E)$.
\ifsodaversion
The proof of Lemma \ref{lem:elem} can be found in the Appendix.
\fi

\begin{lemma}\label{lem:elem}
Let $W\in \cL$ and $E\subseteq \delta^{in}_{D^+}(W)$. If $E$ is $W$-\elem, then $r_W(E)=\min\{f_W(Z)+\varrho_E(Z):\emptyset\ne Z\subseteq W\}$.
\end{lemma}

\ifsodaversion
\else

\begin{proof}By Lemma \ref{lem:ranklefW}, $r_W(E)\le\min\{f_W(Z)+\varrho_E(Z):\emptyset\ne Z\subseteq W\}$. We prove the other direction by induction on the size of $W$. If
$|W|=1$, then equality holds for $Z=W$, because we assumed that
  $r_W(E)<k$. If $|W|>1$, then let $W_1,\dots,W_l$ be the maximal
  members of $\cL[W]-W$.  Since $E$ is $W$-\elem, there is a
  $\cL[W]$-compatible set $\emptyset \neq X \subseteq W$ such that
  $r_{W}(E)=r_{W}^\oplus(\delta^{in}_{E\cup D[W]}(X))$ and
  $E_i:=\delta^{in}_{E\cup D[W]}(X) \cap \delta^{in}(W_i)$ is
  $W_i$-\elem\ for every $i$.  We may assume that $X=\cup_{i=1}^tW_i$
  for some $1\le t\le l$, and thus $r_{W}(E)=\sum_{i=1}^t
  r_{W_i}(E_i)$. By induction, there exist nonempty $Z_i \subseteq
  W_i$ ($i=1,\dots,t$) such that $r_{W_i}(E_i)=f_{W_i}(Z_i) +
  \varrho_{E_i}(Z_i)$.  Let $Z=\cup_{i=1}^t Z_i$. Observe that an arc
  entering $W_i$ but not entering $X$ does not contribute to $f_{W}(Z)
  + \varrho_{E}(Z)$, thus $f_{W}(Z) + \varrho_{E}(Z) =\sum_{i=1}^t
  (f_{W_i}(Z_i) + \varrho_{E_i}(Z_i))=r_{W}(E)$.
\end{proof}
\fi

If a digraph $D'$ is considered instead of $D$, then we speak of $(D',W)$-\elem\ arc sets.
We also extend the notion to arc sets in $D[W]$ entering a specified $\cL[W]$-compatible subset.
For $W\in \cL$ and an $\cL[W]$-compatible subset $X$ of $W$, we say that a set
$E\subseteq \delta^{in}_{D[W]}(X)$ is \textbf{$X$-\elem}\ if $r^\oplus_W(E)<k$ and  $E\cap
\delta^{in}(W')$ is $W'$-\elem\ for every maximal member $W'$ of
$\cL[W]-W$. The following is an easy consequence of Lemma \ref{lem:elem}.

\begin{lemma}\label{lem:elem2}
Let $W\in \cL$ and let $E\subseteq \delta^{in}_{D[W]}(X)$ for some nonempty $\cL[W]$-compatible subset $X$ of $W$.
If $E$ is $X$-\elem, then $r^{\oplus}_W(E)=\min\{f_W(Z):\emptyset\ne Z\subseteq X\}$. \qed
\end{lemma}

Using this lemma, we can finally prove our main result on the minimum size of transversals.
\ifsodaversion
The proof of Theorem   \ref{thm:main} can be found in the Appendix.
\fi
\begin{thm}\label{thm:main}
 If the minimum size of a transversal is $\gamma \geq k$, then $\gamma$ equals
 \begin{equation}
\min_{W\in \cL} \min \{f_W(Z_1)+f_W(Z_2)-k+1:\ \text{$Z_1,Z_2$ are disjoint subsets of $W$}\}.\label{eq:gamma}
 \end{equation}
 \end{thm}
\ifsodaversion

\else

\begin{proof}
By Claim \ref{cl:Z_1Z_2},
if $W \in \cL$ and $Z_1,Z_2$ are nonempty disjoint subsets of $W$, then there is a transversal of size
$f_W(Z_1)+f_W(Z_2)-k+1$, thus $\gamma$ is at most \eqref{eq:gamma} (this is true even if $\gamma < k$).

To show that equality holds for some $W \in \cL$,
let $H$ be a minimum transversal,
and let $D'=D^+-H$. By Theorem
\ref{thm:minarcset}, there exists $W \in \cL$ and an
$\cL[W]$-compatible subpartition $\cX$ of $W$ such that $H\subseteq D[W]$ and $\sum_{X \in
  \cX} r_{D', W}^\oplus (\delta^{in}_{D'[W]}(X))= k(|\cX|-1)-1$. Let us choose a minimum transversal $H$ for which
  $W$ is the smallest possible, and (subject to that) $\cX$ has the smallest possible cardinality;
this implies that $r_{D', W}^\oplus (\delta^{in}_{D'[W]}(X))<k$ for every $X
\in \cX$.
\begin{claim}
$|\cX|=2$. \label{cl:exchange}
\end{claim}
\begin{proof}
Suppose for contradiction that $|\cX| \geq 3$. Then $0<r_{D',
  W}^\oplus (\delta^{in}_{D'[W]}(X))<k$ for every $X \in \cX$;
furthermore, by the assumption $\gamma \geq k$ and Corollary
\ref{cor:gamma<k}, all of these ranks except for at most one were
originally $k$ in $D$. Let $X_0$ be one of the members of \cX\ for
which $r_{D, W}^\oplus (\delta^{in}_{D[W]}(X_0))= k$, and let $X_1$ be
another member.  Let $E_1= \delta^{in}_{D'[W]}(X_1)$, and consider the
following arc exchange operation.

\paragraph{$(X_0,E_1)$-{\sc exchange}} By Lemma \ref{lem:decrease_rank}, there exists an arc $e\in D'[W]$ such that
$r_{D'-e, W}^\oplus (E_1-e)<r_{D', W}^\oplus (E_1)$.
Choose an arbitrary arc $e_0 \in H$ whose head is in $X_0$
(such an arc exists because $r_{D', W}^\oplus (\delta^{in}_{D'[W]}(X_0))<r_{D, W}^\oplus (\delta^{in}_{D[W]}(X_0))$).
Let $H_1=H-e_0+e$.
\bigskip

By the choice of $H$, there is no $W'
\subset W$ such that $H_1$ is a transversal of $\cL[W']$-tight \karb s in $D'_{W'}$.
By the choice of $e$, $H_1$ is still a transversal of $\cL[W]$-tight \karb s, so it
is a minimum transversal. We can apply the exchange operation repeatedly until we
obtain a minimum transversal $H''$ for which $ r_{D'', W}^\oplus (\delta^{in}_{D''[W]}(X_0))= k$, where $D''=D^+-H''$.
At this
point, $\cX-X_0$ is a good subpartition for $H''$ that has fewer
members than $\cX$, in contradiction to the choice of $H$ and $\cX$.
\end{proof}

We obtained that $\cX$ is a subpartition with two members, so $\cX= \{X_1, X_2\}$ and $r_{D', W}^\oplus
(\delta^{in}_{D'[W]}(X_1))+r_{D', W}^\oplus (\delta^{in}_{D'[W]}(X_2))=
k-1$. The next claim shows that $H$ can be modified so that the arc sets in the formula become \elem.

\begin{claim} \label{cl:Hstar}
There is a minimum transversal $H^*$ of $\cL[W]$-tight \karb s such that $\delta^{in}_{D^*[W]}(X_j)$ is $(D^*,X_j)$-\elem\
and $r_{D^*, W}^\oplus (\delta^{in}_{D^*[W]}(X_j))=r_{D', W}^\oplus (\delta^{in}_{D'[W]}(X_j))$
for $j=1,2$ (where $D^*$ denotes $D^+-H^*$).
\end{claim}
\begin{proof}
If $\delta^{in}_{D'[W]}(X_j)$ is $(D',X_j)$-\elem\ for $j=1,2$, then $H$ has the required properties.
Suppose that $\delta^{in}_{D'[W]}(X_j)$ is not $(D',X_j)$-\elem\ . This means that if we recursively compute the rank
of $\delta^{in}_{D'[W]}(X_j)$, then at some point we have to compute a rank $r_{D',W'}(E')$
for some $W'\in \cL[W]-W$ and some $E'\subseteq \delta_{D'}^{in}(W')$, but the smallest $\cL[W']$-compatible
subpartition $\cY$ that determines $r_{D',W'}(E')$ has at least two members.

Since $0<r_{D',W'}(E')<k$, we have $0<r_{D',
  W'}^\oplus(\delta^{in}_{E'\cup D'[W']}(Y))<k$ for every $Y \subseteq
\cY$. Let $E=E'\cup (H \cap \delta_{D}^{in}(W'))$. By the assumption
$\gamma \geq k$ and Corollary \ref{cor:gamma<k}, we know that
$r_{W'}^\oplus(\delta^{in}_{E \cup D[W']}(Y))=k$ for all but at most
one member of $\cY$; let $Y_0$ be a member for which it is $k$, and
let $Y_1$ be another member.  Let $E_1=\delta^{in}_{E' \cup
  D'[W']}(Y_1)$.  By the same argument as in the proof of Claim
\ref{cl:exchange}, a $(Y_0,E_1)$-{\sc exchange} operation results in a
transversal of the same size as $H$, for which
$r_{D',W'}^\oplus(\delta^{in}_{E' \cup D[W']}(Y_0))$ increases by one.
By applying the exchange operation repeatedly, we eventually obtain a
transversal $H''$ such that $|H''|=|H|$ and $r_{D'',W'}^\oplus
(\delta^{in}_{E'' \cup D''[W']}(Y_0))= k$, where $E''= E - H''$ and
$D''=D^+-H''$. At this point, $\cY-Y_0$ also determines the rank
$r_{D'',W'}(E'') = r_{D',W'}(E')$, and has fewer members than $\cY$.

By repeating this procedure, we eventually obtain a transversal $H^*$ which satisfies the claimed properties.
\end{proof}

Let $H^*$ be the minimum transversal given by Claim \ref{cl:Hstar}. By Lemma \ref{lem:elem2}, there is a nonempty set
$Z_j \subseteq X_j$ such that $r_{D^*, W}^\oplus
(\delta^{in}_{D^*[W]}(X_j))=f_{D^*,W}(Z_j)$, for both $j=1,2$. Thus $f_{D^*,W}(Z_1)+f_{D^*,W}(Z_2)=k-1$.
Since the removal of an arc from $D$ can decrease
$f_{W}(Z_1)+f_{W}(Z_2)$ by at most one, we have $\gamma=|H^*| \geq f_{W}(Z_1)+f_{W}(Z_2)-k+1$. As the reverse
inequality has already been proved, this completes the proof of the theorem.
\end{proof}
\fi

The theorem not only characterizes the minimum size of transversals if $\gamma \geq k$, but also guarantees
the existence of minimum transversals that have a special structure.

\begin{cor}\label{cor:mintrans}
 Suppose that $\gamma \geq k$, and let $(W,Z_1,Z_2)$ be minimizers of \eqref{eq:gamma}.
  Let
 \[E_j=\{e \in D[W]: e \in \delta^{in}(Z_j),\ e \notin \delta^{out}(W')
\text{ if $W' \in \cL[W]$ and $W' \cap Z_j \neq \emptyset$}\}\quad (j=1,2).\]
Then every arc set $H \subseteq E_1 \cup E_2$ of size $|E_1 \cup E_2|-k+1$ is a minimum transversal of the family of
\cL-tight \karb s.
\end{cor}
\begin{proof}
 By Theorem \ref{thm:main}, $\gamma=f_W(Z_1)+f_W(Z_2)-k+1$, so $|H|=\gamma$. Let $D'=D-H$; by definition,
 $f_{D',W}(Z_1)+f_{D',W}(Z_2)=f_W(Z_1)+f_W(Z_2)-|H|$, thus $f_{D',W}(Z_1)+f_{D',W}(Z_2)=k-1$. According to
 Claim \ref{cl:Z_1Z_2}, no \cL-tight \karb\ exists in $D'$, so $H$ is a transversal.
\end{proof}

Using this, we can give a polynomial time algorithm if $k$ is fixed. We check if there is a transversal of size at
most $k-1$ by brute force search. If there is none, then we can use the algorithm \verb#covering_tight_arborescences#  in \cite{mincostarb} to
compute $\min_{W\in\cL}(\min\{f_W(Z_1)+f_W(Z_2):Z_1, Z_2$ are nonempty, disjoint subsets of $W\})$
and minimizers $(W,Z_1,Z_2)$ in polynomial time. We can also determine the arc sets $E_1,E_2$ as in Corollary \ref{cor:mintrans}, so
we can find a transversal of minimum size.

\section{Conclusion}

As the example in Figure \ref{fig:Tamaspelda} shows, the minimum size of a transversal can be smaller than \eqref{eq:gamma}.
To make further progress on the problem, this case should be better understood.
As mentioned at the end of Section \ref{sec:variants}, it can be decided in polynomial time using a weighted
matroid intersection algorithm whether there is an $\cL$-tight \karb; in this sense, the case $\gamma=0$ is well-understood in terms of general
matroid techniques. However, such techniques do not suffice for higher $\gamma$,
as the transversal problem for general matroid intersection (and even for general matroids) is NP-hard.
The algorithm presented in Section \ref{sec:blocking} sidesteps this problem by simply checking for every arc
subset of size at most $k$ whether it is a transversal; this of course means that the algorithm is not even
fixed-parameter tractable for the parameter $k$.
One possible approach to improve this would be to generalize the subpartition-finding algorithms of \cite{kmincostarb} to laminar families.

\bibliographystyle{amsplain} \bibliography{bkmincost}

\ifsodaversion

\newpage

\section{Appendix}

\subsection{Proofs from Section \ref{sec:variants}}

\begin{proof}[Proof of Theorem \ref{thm:versions}]
Problem \ref{prob:2} reduces to Probem \ref{prob:1} by deleting all arcs entering node $s$ from the input digraph. For the other direction, consider an instance $D,k,c$ of Problem \ref{prob:1}, and let $\alpha=|A|+k$, $\beta=\sum_{a\in A}c(a)+1$. Let $(D^+,c^+)$ be the $(\alpha,\beta)$-extension of $(D,c)$. In the instance of Problem \ref{prob:2}
given by $(D^+,k,c^+,s)$, the minimum cost \krarb s
naturally correspond to minimum $c$-cost \karb s in $D$
(since they contain exactly $k$ arcs leaving $s$ because of the value of $\beta$). Moreover, the minimum size of a transversal is at most $|A|$ as $A$ itself is a transversal.
This shows that every minimum transversal is a subset of $A$.
\end{proof}

\begin{proof}[Proof of Corollary \ref{cor:slackness}]
Consider the LP $\min\{cx: x\in \Rset^A,\ 0\le x\le 1,\ \varrho_x(Z)\ge
k$ for ever non-empty $Z\subseteq V-s\}$. By Theorem
\ref{thm:karbhull}, this has an integer optimal solution, which is a
minimum cost \krarb. Let $y^*,z^*$ be an optimal solution of the
dual
\begin{align*}\label{lp:dual}
\max \sum_{\emptyset \neq Z \subseteq V-s} ky_Z-\sum_{a\in A}z_a\\
 y\in \Rset_+^{2^{V-s}-\{\emptyset\}},z\in \Rset_+^A  \\
 \sum_{Z:
  a\in \delta^{in}(Z)}y_Z-z_a  \le c_a \text{ for every } a\in A.
\end{align*}
  We can assume that the support of $y^*$ is a laminar family
$\cL\subseteq 2^V$ by Theorem \ref{thm:karbhull}. The complementary slackness conditions show that a
feasible primal solution $x^*$ is optimal if and only if the following
three conditions hold.
\begin{enumerate}
\item $x^*_a=0$ for every $a\in A$ with $\sum_{Z: a\in \delta^{in}(Z)}y^*_Z-z^*_a< c_a$ (\textbf{forbidden} arcs),
\item $\varrho_{x^*}(W)=k$ for every $W\in \cL$, and
\item $x^*_a=1$ for every $a\in A$ with $z^*_a>0$ (\textbf{mandatory} arcs).
\end{enumerate}

By denoting the forbidden arcs by $A_0$ and the mandatory arcs by $A_1$ we obtain the required structure.
\end{proof}

\subsection{Proofs of Theorem \ref{thm:indepkarb} and Lemma \ref{lem:rankin}}

\begin{proof}[Proof of Theorem \ref{thm:indepkarb}]
The necessity of \eqref{eq:subpart} is clear: if $F\subseteq A$ is a
\indep\ \karb\ and \cX\ is a subpartition of $V$, then $k(|\cX|-1)\le
\sum_{X\in \cX} \varrho_F(X)\le \sum_{X\in
  \cX}r^\oplus(\delta^{in}_{D}(X))$.  In order to prove sufficiency,
let $M_1=(A,r_1)$ be  $k$ times the circuit matroid of the underlying undirected graph of $D$. Note that condition \eqref{eq:subpart}
implies that $D$ contains $k$ edge-disjoint spanning trees, thus
$r_1(A)=k(|V|-1)$. For every $v\in V$, let $M_v'=(\delta^{in}_D(v),
r'_v)$ be the $k$-shortening of $M_v$, that is $r'_v(E)=\min\{r_v(E),
k\}$ for every $E\subseteq \delta^{in}_D(v)$. Let furthermore $M_2=(A,
r_2)$ be the direct sum of the matroids
$M'_v$. Observe that $F\subseteq A$ is a \indep\ \karb\ in $D$ if and
only if $F$ is a common independent set of $M_1$ and $M_2$ and has
size $k(|V|-1)$. By Edmonds' matroid intersection theorem \cite{edmonds},
such an $F$ exists if and only if
\begin{equation}\label{eq:matroidint}
r_1(E)+ r_2(A-E)\ge k(|V|-1)\mbox{ for every } E \subseteq A.
\end{equation}
We show that condition \eqref{eq:subpart} implies
\eqref{eq:matroidint}. Suppose that \eqref{eq:matroidint} fails for
some $E$. Clearly, we can assume that $E$ is closed in $M_1$ and
$M_1|E$ does not contain bridges (a \textbf{bridge} in a matroid is an element
that is contained in every base).
\begin{claim}
If $E\subseteq A$ is closed in $M_1$ and $M_1|E$ does not contain
bridges, then there exists a partition $\cY$ of $V$
such that $r_1(D[Y])=k(|Y|-1)$ for every $Y\in \cY$ and
$E=\cup_{Y\in \cY}D[Y]$.
\end{claim}
\begin{proof}
We say that a non-empty $Y\subseteq V$ is \textbf{tight} (with respect
to $E$) if $r_1(E[Y])= k(|Y|-1)$. In other words, $Y$ is tight if
$E[Y]$ contains $k$ edge-disjoint  trees, each spanning $Y$. For example, sets of
size 1 are tight. If $Y_1, Y_2$ are both tight and $Y_1\cap Y_2\ne
\emptyset$, then $Y_1\cup Y_2$ is tight, too. To prove this, let
$T_1\subseteq E$ be a tree spanning $Y_1$ and $T_2\subseteq E$ be a
tree spanning $Y_2$, and observe that $T_1$ can be extended to a tree
spanning $Y_1\cup Y_2$ using the edges of $T_2-E[Y_1]$. Therefore let
$\cY$ be the partition of $V$ consisting of the maximal tight
sets. Since $E$ is closed in $M_1$, it contains every arc of $D$ that
is induced in some $Y\in \cY$. Let $G'=(V', E')$ be the graph
obtained from $(V,E)$ after contracting every $Y\in \cY$ into a node
$y$. We claim
that $i_{G'}(Z)<k(|Z|-1)$ for every $Z\subseteq V'$ with $|Z|\ge
2$. Suppose not and take an inclusionwise minimal set
$Z$ with $i_{G'}(Z) \geq k(|Z|-1)$. Then $G'[Z]$ contains $k$ edge-disjoint spanning trees by the theorem of Tutte and Nash-Williams \cite{tutte1961problem},
which contradicts the maximality of the tight sets in
$\cY$. This implies that the bases of $M_1|E$ contain every arc of $E$
going between different members of the partition $\cY$. But since
$M_1|E$ does not contain bridges, $E=\cup_{Y\in \cY}D[Y]$, as claimed.
\end{proof}
Consider the partition \cY\ defined in the claim above, and observe that
$r_1(\cup_{Y\in \cY}D[Y]) + r_2(\cup_{Y\in \cY}\delta^{in}_D(Y))
=k(|V|-|\cY |) + \sum_{Y\in \cY} r_2(\delta^{in}_D(Y)) < k(|V|-1) $, thus $\sum_{Y\in \cY} r_2(\delta^{in}_D(Y)) < k(|\cY |-1)$. Let
$\cX=\{Y\in \cY: r_2(Y)<k \}$ and note that $\sum_{X\in \cX}
r_2(\delta^{in}_D(X)) < k(|\cX|-1)$ holds as well.  But
$r_2(\delta^{in}_{D}(X))=r^\oplus(\delta^{in}_{D}(X))$ for every $X\in \cX$, thus we get a contradiction
with \eqref{eq:subpart}.
\end{proof}

\begin{proof}[Proof of Lemma  \ref{lem:rankin}]
It is clear that \eqref{it:1} implies \eqref{it:2}.
Let us prove that \eqref{it:2} implies \eqref{it:1}.  Let $D'=I\cup
D[V-s]$.
We will prove that there exists a \indep\  \karb\ in $D'$ by applying Theorem
\ref{thm:indepkarb}. Suppose that $\sum
\{r^\oplus(\delta^{in}_{D'}(X)): X\in \cX\} < k(|\cX|-1)$ for some
subpartition \cX. Note that we can assume
$r^\oplus(\delta^{in}_{D'}(X)) < k$ for every member $X$ of \cX, and
clearly $|\cX|>1$ has to hold. Therefore there must exist a member
$X\in\cX$ with $s\notin X$ and $r^\oplus(\delta^{in}_{D'}(X))<k$,
contradicting \eqref{it:2}.

Next we show that \eqref{it:1} implies \eqref{it:3}.  If $F\subseteq
A$ is a \indep\  \krarb[s] with $I=F\cap \delta^{out}_D(s)$,
$E\subseteq \delta^{out}_D(s)$, and $X\subseteq V-s$, then $k\le
\varrho_F(X) = \varrho_{F\cap E}(X)+ \varrho_{F- E}(X) \le |F\cap E| +
r^\oplus(\delta^{in}_{D-E}(X)) = |I\cap E| +
r^\oplus(\delta^{in}_{D-E}(X))$.  Finally, we show that \eqref{it:3}
implies \eqref{it:2}. Take some non-empty $X\subseteq V-s$, let
$E=(\delta_D^{out}(s)\cap \delta_D^{in}(X))-I$ and apply
the property in \eqref{it:3} for $X$ and $E$ to obtain \eqref{it:2}.
\end{proof}

\subsection{Proof of Theorem \ref{thm:supermod}}

For the proof of Theorem \ref{thm:supermod} we need the following claims.

\begin{claim}\label{cl:rplus}
Let $E_1, E_2\subseteq \delta_D^{out}(s)$ and  $X_1, X_2\in
V-s$ be arbitrary, then
\begin{equation}
r^\oplus(\delta^{in}_{D-E_1}(X_1))+ r^\oplus(\delta^{in}_{D-E_2}(X_2)) \ge r^\oplus(\delta^{in}_{D-(E_1\cup E_2)}(X_1\cup X_2)) + r^\oplus(\delta^{in}_{D-(E_1\cap E_2)}(X_1\cap X_2)).
\end{equation}
\end{claim}
\begin{proof}
By the properties of the direct sum, it is enough to show the following for an arbitrary $v\in V$, where $\Delta$ denotes $\delta^{in}_D(v)$.
\begin{equation}\label{eq:rvsub}
r_v(\delta^{in}_{\Delta-E_1}(X_1))+ r_v(\delta^{in}_{\Delta-E_2}(X_2)) \ge r_v(\delta^{in}_{\Delta-(E_1\cup E_2)}(X_1\cup X_2)) + r_v(\delta^{in}_{\Delta-(E_1\cap E_2)}(X_1\cap X_2)).
\end{equation}
If $v\notin X_1\cup X_2$, then there is nothing to prove, every term is
zero on both sides of \eqref{eq:rvsub}. If $v\in X_1-X_2$, then the
second term is zero on both sides of \eqref{eq:rvsub}, and the inequality
$r_v(\delta^{in}_{\Delta-E_1}(X_1))\ge
r_v(\delta^{in}_{\Delta-(E_1\cup E_2)}(X_1\cup X_2))$ is
implied by the mononicity of $r_v$. Clearly, the case $v\in X_2-X_1$
is analogous, therefore assume $v\in X_1\cap X_2$. Observe that \eqref{eq:rv1} and \eqref{eq:rv2} holds. For an illustration, see Figure \ref{fig:rv}.
\begin{eqnarray}\label{eq:rv1}
\delta^{in}_{\Delta-E_1}(X_1)\cap \delta^{in}_{\Delta-E_2}(X_2) = \delta^{in}_{\Delta-(E_1\cup E_2)}(X_1\cup X_2)\\
\label{eq:rv2} \delta^{in}_{\Delta-E_1}(X_1)\cup \delta^{in}_{\Delta-E_2}(X_2) = \delta^{in}_{\Delta-(E_1\cap E_2)}(X_1\cap X_2).
\end{eqnarray}
\begin{figure}
\begin{center}
\scalebox{0.8}{\input{rv.pdf_t}}
\caption{An illustration for proving \eqref{eq:rv1} and \eqref{eq:rv2}. The arcs of $(E_1-E_2)\cap \delta_D^{in}(v)$ are coloured blue, those in $(E_2-E_1)\cap \delta_D^{in}(v)$ are red, and those in $(E_1\cap E_2)\cap \delta_D^{in}(v)$ are magenta. That is, $\Delta-E_1$ is the set of arcs in the figure that are neither blue, nor magenta, etc.}
\label{fig:rv}
\end{center}
\end{figure}
This, together with the submodularity of $r_v$, finishes the proof.
\end{proof}


\newcommand{\dels}{\ensuremath{\delta_D^{out}(s)}}

Let us introduce the following notation. For a set $E\subseteq
\delta_D^{out}(s)$, let $X_E\subseteq V-s$ be an arbitrary subset that
attains the maximum in the definition \eqref{eq:pZ} of $p(E)$ (that
is, $X_E\ne \emptyset$ and $p(E)=k-r^\oplus(\delta^{in}_{D-E}(X_E))$).

\begin{claim}\label{cl:headin}
If $E\subseteq \dels$ is non-separable, then $X_E$ contains the head of every arc of $E$.
\end{claim}
\begin{proof}
Suppose not and let $E_1\subsetneq E$ be the subset of those arcs which
have their head in $X_E$. Then
$p(E)=k-r^\oplus(\delta^{in}_{D-E}(X_E))=k-r^\oplus(\delta^{in}_{D-E_1}(X_E))\le
p(E_1)$. But then
$p(E)\le p(E_1)+p(E-E_1)$ by the non-negativity of $p$, contradicting
the non-separability of $E$.
\end{proof}

\begin{proof}[Proof of Theorem \ref{thm:supermod}]
Let $E_1, E_2\subseteq\delta^{out}_D(s) $ be non-separable sets so
that $E_1\cap E_2\ne \emptyset$. By Claim \ref{cl:headin}, $X_i=X_{E_i}$ contains the head of each arc of $E_i$ for both $i=1,2$.
This implies that $X_1\cap X_2\ne \emptyset$, and  Claim \ref{cl:rplus} gives
\begin{multline*}
p(E_1)+p(E_2)= \sum_{i=1,2}k-r^\oplus(\delta^{in}_{D-E_i}(X_i))\le \\
2k-\left(r^\oplus(\delta^{in}_{D-(E_1\cup E_2)}(X_1\cup X_2)) + r^\oplus(\delta^{in}_{D-(E_1\cap E_2)}(X_1\cap X_2))\right)\le
p(E_1\cap E_2)+p(E_1\cup E_2).
\end{multline*}
\end{proof}

\subsection{Proofs of  Claim  \ref{cl:wedge} and Corollary \ref{cor:indepmatroid}}

\begin{proof}[Proof of Claim  \ref{cl:wedge}]
Let $E\subseteq \delta^{out}_D(s)$ and let \cH\ be a partition of
$E$ that gives $p^\wedge(E)=\sum \{p(H): H\in \cH\}$ and, subject to
this, $|\cH|$ is minimal. Clearly, every $H\in \cH$ is non-separable.
We claim that $\{X_H: H\in \cH\}$ is a subpartition of $V-s$. If there
exist $H_1, H_2\in \cH$ so that $X_{H_1}\cap X_{H_2}\ne \emptyset$,
then (by Claim \ref{cl:rplus}) $p(H_1)+p(H_2)\le 2k-(r^\oplus(\delta^{in}_{D-(H_1\cap
  H_2)}(X_{H_1}\cap X_{H_2})) + r^\oplus(\delta^{in}_{D-(H_1\cup
  H_2)}(X_{H_1}\cup X_{H_2})))\le p(H_1\cup H_2) + p(\emptyset) =
p(H_1\cup H_2)$, therefore $\cH'=\cH-\{H_1, H_2\} + \{H_1\cup H_2\}$
also gives $p^\wedge(E)=\sum \{p(H): H\in \cH'\}$, contradicting our
choice of \cH.
\end{proof}

\begin{proof}[Proof of Corollary \ref{cor:indepmatroid}]
Consider the function $p^\wedge$ defined by \eqref{eq:pwedge2}.
By Theorem \ref{thm:wedge}, $p^\wedge$ is monotone increasing and
supermodular, and $P=\{x\in \Rset^{\delta^{out}(s)}:
x(E)\ge p^\wedge(E)\mbox{ for every }E\subseteq
\delta^{out}_D(s)\}$ if $P$ is non-empty. Thus $B=\{x\in P: x( \delta^{out}_D(s)) = k\}$
is not empty if and only if $P\ne \emptyset$ and
$p^\wedge(\delta^{out}_D(s))=k$, that is, if and
only if both \eqref{it:Pnempty} and \eqref{it:Bnempty} hold.
Since the fully supermodular function describing the base polyhedron $B$ is $p^\wedge$, it is the co-rank function of the matroid $M^s$, and its rank function is given by the formula
\begin{multline*}
 r^s(E)=p^\wedge(\delta^{out}_D(s))- p^\wedge(\delta^{out}_D(s)-E)=k-p^\wedge(\delta^{out}_D(s)-E) \\= \min\{\sum_{X\in \cX} r^\oplus(\delta^{in}_{E\cup D[V-s]}(X)) - k(|\cX|-1): \cX \mbox{ is a subpartition of }V-s\}.
 \end{multline*}

\end{proof}

\subsection{Proof of Theorem \ref{thm:MF}}

\begin{proof}[Proof of Theorem \ref{thm:MF}]
We recursively show that the family $\cB_W$ indeed defines a matroid $M_W$ for every $W\in \cL$. For
the singletons $\{v\}\in \cL$ it is clear that $M_{\{v\}}$ is the
uniform matroid of rank $k$ on ground set $\delta_{D^+}^{in}(v)$. Let $W\in
\cL$ be a non-singleton, and assume that $M_{W'}$ has already been
defined for every $W'\in \cL$ that is a proper subset of $W$. Let $W_1, W_2,
\dots , W_l$ be the maximal members of $\cL[W]-W$, and let us contract each
$W_i$ into a single node $w_i$ ($i=1,2,\dots,l$).  Let
 $\Wc = W/\{W_1, W_2, \dots , W_l\}$ be the set obtained
from $W$ by these contractions, and
similarly, for a subgraph $(W+s_W, E)$ of $D_W$ we use the
notation $\Ec = E/\{W_1, W_2, \dots , W_l\}$ to mean the graph
obtained from $(W+s_W, E)$ by the contractions (and deletion of the loops that arise).
In particular, let $\Dc=D_{W}/\{W_1, W_2,
\dots , W_l\}$. The matroids $M_{W_i}$ naturally give rise to matroids
$M_{w_i}=(\delta^{in}_{\Dc}(w_i), r_{w_i})$ for
every $i$; let $\cM=\{M_{w_1}, \dots, M_{w_l}\}$.
\begin{claim}
If $F\subseteq A(D_{W})$ is an $\cL[W]$-tight
\krarb[s_W], then $\Fc$ is
\cM-\indep. Conversely, if
$F'\subseteq \Dc$ is an \cM-\indep\  \krarb[s_W] in $\Dc$ and
$|\delta_{F'}^{out}(s_W)|=k$, then there exists an $\cL[W]$-tight
\krarb[s_W] $F\subseteq A(D_{W})$ such that $\Fc=F'$.
\end{claim}
\begin{proof}
The first statement is clear from the definition of the matroids
$M_{w_i}$. For the other direction, let $F'\subseteq \Dc$ be an
\cM-\indep\ \krarb[s_W] in $\Dc$, such that $|\delta_{F'}^{out}(s_W)|=k$.
Consider $F'$ as a
subgraph of $D_W$, and note that $\delta^{in}_{F'}(W_i)$ is a base of
$M_{W_i}$ for every $i$. By the definition of $M_{W_i}$,
$\delta^{in}_{F'}(W_i)$ can be extended to an $\cL[W_i]$-tight
arborescence $F_i$ in $D_{W_i}$ for every $i$. The \krarb[s_W] $F=F'\bigcup \cup_iF_i$ is $\cL[W]$-tight
and $\Fc= F'$, as required.
\end{proof}

The claim implies that $\cB_W$ consists of the arc sets of size $k$ that can be obtained as the arcs incident to $s_W$
of an \cM-\indep\ \krarb[s_W], so the statement of the theorem
follows from Corollary \ref{cor:indepmatroid}.
\end{proof}

\subsection{Proof of Theorem \ref{thm:optimalroot}}

In order to prove Theorem \ref{thm:optimalroot} we introduce the following transformation.

\newcommand{\mandconstr}{{\sc mandatory arc transformation}}
\begin{figure}
\begin{center}
\scalebox{0.8}{\input{mand_constr.pdf_t}}
\caption{An illustration for the \mandconstr.}
\label{fig:mand}
\end{center}
\end{figure}

\paragraph{\mandconstr} Given a digraph $D=(V,A)$, a node $s\in V$, an
arc $a=uv$ (where $u,v\in V-s$), and a laminar family $\cL\subseteq
2^{V-s}$, we construct a digraph $D'=(V+x_a, A-a+B_a)$, where
$B_a=\{ux_a, x_av\}\cup\{k-1$ parallel copies of $vx_a\}$.  Let
furthermore $\cL'\subseteq 2^{V+x_a}$ be defined as $\cL'=\{W\in \cL:
v\not\in W\}\cup\{W+x_a: v\in W\in \cL\}$ (note that $\{v\}\in \cL$ implies that $\{x_a,v\}\in \cL'$).
See Figure \ref{fig:mand} for an illustration. It is
easy to check that $\cL'$ is laminar.

\begin{claim}\label{cl:phi}
For an \cL-tight \krarb\ $F\subseteq A$ containing $a=uv$, let
$\phi(F) = F-a+B_a$. Then $\phi$ is a bijection between \cL-tight
\krarb{s} containing $a$ in $D$ and $\cL'$-tight \krarb{s} in $D'$.
\end{claim}
\begin{proof}
First we show that if $F\subseteq A$ is an \cL-tight \krarb\  containing $a$, then $\phi(F)$
is an $\cL'$-tight \krarb\ in $D'$. Let $F_1, F_2,
\dots, F_k$ be a decomposition of $F$ into $k$ $s$-rooted
arborescences and assume that $a\in F_1$. Let
$F_1'=F_1-a+\{ux_a,x_av\}$, and let $F_i'=F_i$ plus a copy of the arc
$vx_a$ for every $i=2,\dots, k$. Then $F_1', F_2', \dots, F_k'$ is a
decomposition of $\phi(F)$ into $k$ $s$-rooted arborescences in $D'$,
so $\phi(F)$ is indeed an \krarb\ in $D'$. Furthermore, $\phi(F)$
is $\cL'$-tight, as $\phi(F)[\{x_a,v\}]$ is a \karb, and the indegree
of any other set $W\in \cL'$ in the subgraph $\phi(F)$ is $k$.

For the other direction, let $F'\subseteq A'$ be an arbitrary
$\cL'$-tight \krarb\ in $D'$. Since $F'[\{x_a,v\}]$ is a \karb\ and
$\varrho_{F'}(x_a)=k$, $B_a\subseteq F'$ must hold. Let
$F=F'-B_a+a$; we show that $F$ is a \cL-tight \krarb\ in $D$ --
since $a\in F$ and $F'=\phi(F)$, this completes the proof. Let $F_1',
F_2', \dots, F_k'$ be a decomposition of $F'$ into $k$ $s$-rooted
arborescences in $D'$, and assume that $ux_a\in F_1'$. Then clearly
$x_av$ is in $F_1'$ too, so $F_1'-\{ux_a,x_av\}+a, F_2'-vx_a, F_3'-vx_a,
\dots, F_k'-vx_a$ is a decomposition of $F$ into $k$ $s$-rooted
arborescences in $D$. The \cL-tightness of $F$ can be shown similarly.
\end{proof}

Using this transformation we can now prove Theorem \ref{thm:optimalroot}.

\begin{proof}[Proof of Theorem \ref{thm:optimalroot}]
Given a digraph $D=(V, A)$ and a cost function $c:A\to \Rset$, let $\alpha=k+1$,
$\beta=\sum_{a\in A}c(a)+1$, and let $(D^+,c^+)$ be the $(\alpha,\beta)$-extension of $(D,c)$.
By previous remarks, optimal \karb{s} in $D$ and optimal \karb{s} in $D^+$
correspond to each other in a natural way (and \karb{s} in $D^+$ are
rooted at $s$).  By Corollary \ref{cor:slackness}, there exists a laminar family
$\cL\subseteq 2^V$ and two disjoint sets $A_0, A_1\subseteq A^+$, such
that a \karb\ $F\subseteq A^+$ is
optimal if and only if $A_1\subseteq F\subseteq A^+-A_0$ and $F$ is
$\cL$-tight. Due to symmetry, $A_1$ contains either all or none of
the parallel arcs between $s$ and a given node $v\in V$. Since there are $k+1$
parallel arcs, the former is impossible, so $A_1\subseteq A$.

Starting with $D^+-A_0$, repeat the \mandconstr\  above for
every $a\in A_1$, to obtain $D'=(V+s+\{x_a:a\in A_1\}, A^+-(A_0\cup
A_1)+\cup_{a\in A_1}B_a))$ and the laminar family $\cL'\subseteq
2^{V+\{x_a:a\in A_1\}}$. For any $\cL$-tight \krarb\ $F\subseteq A^+$
with $A_1\subseteq F\subseteq A^+-A_0$, let $\phi(F)=F-A_1+\cup_{a\in
  A_1}B_a$. By Claim \ref{cl:phi}, $\phi$ defines a bijection between
$\cL$-tight \krarb{s} in $D^+-A_0$ containing $A_1$ and $\cL'$-tight
\krarb{s} in $D'$. By Corollary \ref{cor:matr}, the family
$\{I\subseteq \delta^{out}_{D'}(s): |I|=k$ and $I$ is contained in a   $\cL'$-tight
\krarb\ of $D'\}$ is the family of bases of a matroid. This implies that the convex hull of root vectors of
optimal $k$-arborescences in $D$ is a base polyhedron.
\end{proof}

\subsection{Proofs from Section \ref{sec:blocking}}

\begin{proof}[Proof of Lemma \ref{lem:rank1}]
Let $E'=E-e$. The inequalities $r_{D',W}(E')\leq r_W(E')\leq r_W(E)$ follow from the definition of the rank. We prove the remaining
inequality by induction on the size of $\cL[W]$; it is clearly true if $W$ is a singleton.
Otherwise, by Corollary \ref{cor:recursive}, there is an $\cL[W]$-compatible subpartition $\cX$ of $W$ that determines $r_{D',W}(E')$, i.e.\
$r_{D',W}(E')=\sum_{X \in \cX} r_{D', W}^\oplus (\delta^{in}_{E' \cup D'[W]}(X))- k(|\cX|-1)$.
We know by induction that
$r_{D', W}^\oplus (\delta^{in}_{E' \cup D'[W]}(X))\geq r_W^\oplus (\delta^{in}_{E \cup D[W]}(X))-1$ for every $X \in \cX$,
and the ranks are different for at most one member of $\cX$,
since $e\in D_{W_i}$ for at most one $W_i$. This proves the inequality because
$r_W(E) \leq \sum_{X \in \cX} r_W^\oplus (\delta^{in}_{E \cup D[W]}(X))- k(|\cX|-1)$.
\end{proof}

\begin{proof}[Proof of Theorem \ref{thm:minarcset}]
First note that $|W|>1$, since $H\subseteq A$. As $H \cap D_W$ is a transversal of $\cL[W]$-tight \krarb[s_{W}]s
(and hence of $\cL$-tight \karb s), minimality of $H$ implies that $H\subseteq D_W$.
Let $W_1, \dots, W_l$ be the maximal members of $\cL[W]-W$. By the choice of $W$, $D'_{W_i}$ contains an
$\cL[W_i]$-tight \krarb[s_{W_i}] for every $i$, thus $r_{D', W}^\oplus$ is well-defined.

Since $D'_{W}$ does not contain an $\cL[W]$-tight \krarb[s_{W}],
\eqref{it:Pnempty} or \eqref{it:Bnempty} fails to hold in Corollary \ref{cor:indepmatroid} for $r_{D', W}^\oplus$.
Suppose that $r_{D', W}^\oplus(\delta^{in}_{D'_W}(X))<k$ for some $\cL[W]$-compatible subset
$X \subseteq W$. Then there is a set $W_i$ such that
$r_{D',W_i}(\delta^{in}_{D'}(W)\cap \delta^{in}_{D'}(W_i))<k$. However, since we did not delete any arc leaving $s$, and already the arcs going from $s$ to $W_i$ have rank $k$ in $M_{W_i}$, we get (by monotonicity of $r_{W_i}$) that $r_{W_i}(\delta^{in}_{D'}(W)\cap \delta^{in}_{D'}(W_i))=k$, a contradiction.


Thus \eqref{it:Bnempty} fails to hold in Corollary \ref{cor:indepmatroid}, that is,
$\sum_{X \in \cX} r_{D', W}^\oplus (\delta^{in}_{D'[W]}(X))< k(|\cX|-1)$ for some
$\cL[W]$-compatible subpartition $\cX$ of $W$. As $H$ is inclusionwise minimal and the removal of an arc can
decrease a rank by at most one according to Lemma \ref{lem:rank1}, the left hand side must be equal to
$k(|\cX|-1)-1$. Since the formula involves only arcs in $D'[W]$, minimality also implies that $H \subseteq D[W]$.
\end{proof}

\begin{proof}[Proof of Lemma \ref{lem:decrease_rank}]
The proof is by induction on $|W|$; the claim is clearly true if $W$ is a singleton.
Let $W_1, \dots, W_l$ be the maximal members of $\cL[W]-W$.
By Corollary \ref{cor:recursive}, there exists an $\cL[W]$-compatible subpartition \cX\ of $W$ that determines $r_W(E)$.
We can choose a set $X\in \cX$ and an index $i$  for which $W_i\subseteq X$ and
\[0<r_{W_i}(\delta^{in}_{E \cup D[W]}(X)\cap \delta^{in}_{E \cup D[W]}(W_i))<k.\]
Let $\Delta$ denote $\delta^{in}_{E \cup D[W]}(X)\cap \delta^{in}_{E \cup D[W]}(W_i)$. By induction, there is an arc
$e\in \Delta \cup D[W_i]$ such that either $D'_{W_i}$ does not contain an $\cL[W_i]$-tight \krarb[s_{W_i}]
(where $D'$ is the digraph obtained by removing $e$),
or $r_{D',W_i}(\Delta-e)=r_{W_i}(\Delta)-1$.
The latter possibility means that $r_{D',W}(E-e)< r_W(E)$; on the other hand, the rank can decrease by at
most one by Lemma \ref{lem:rank1}.
\end{proof}

\begin{proof}[Proof of Lemma \ref{lem:ranklefW}]
By the definition of the rank, there exists an $\cL[W]$-tight \krarb[s_W] $F$ such that $|F\cap E|= r_W(E)$.
We apply the tail-relocation operation
described in the proof of Claim \ref{cl:Z_1Z_2}; let $D'$ be the modified digraph, and let $F'$ be the $\cL[W]$-tight \krarb[s_W] obtained from $F$.
On one hand, $f_W(Z)=f_{D',W}(Z)=\varrho_{D'[W]}(Z)$. On the other hand,
\[k \leq \varrho_{F'}(Z) \leq \varrho_{D'[W]}(Z)+\varrho_{E\cap F}(Z)+ \varrho_{F-E}(W)\leq \varrho_{D'[W]}(Z)+\varrho_{E}(Z)+(k-r_W(E)),\]
so $r_W(E) \leq \varrho_{D'[W]}(Z)+ \varrho_{E}(Z)=f_W(Z)+\varrho_E(Z)$, as required.
\end{proof}

\begin{proof}[Proof of Lemma \ref{lem:elem}]
By Lemma \ref{lem:ranklefW}, $r_W(E)\le\min\{f_W(Z)+\varrho_E(Z):\emptyset\ne Z\subseteq W\}$. We prove the other direction by induction on the size of $W$. If
$|W|=1$, then equality holds for $Z=W$, because we assumed that
  $r_W(E)<k$. If $|W|>1$, then let $W_1,\dots,W_l$ be the maximal
  members of $\cL[W]-W$.  Since $E$ is $W$-\elem, there is a
  $\cL[W]$-compatible set $\emptyset \neq X \subseteq W$ such that
  $r_{W}(E)=r_{W}^\oplus(\delta^{in}_{E\cup D[W]}(X))$ and
  $E_i:=\delta^{in}_{E\cup D[W]}(X) \cap \delta^{in}(W_i)$ is
  $W_i$-\elem\ for every $i$.  We may assume that $X=\cup_{i=1}^tW_i$
  for some $1\le t\le l$, and thus $r_{W}(E)=\sum_{i=1}^t
  r_{W_i}(E_i)$. By induction, there exist nonempty $Z_i \subseteq
  W_i$ ($i=1,\dots,t$) such that $r_{W_i}(E_i)=f_{W_i}(Z_i) +
  \varrho_{E_i}(Z_i)$.  Let $Z=\cup_{i=1}^t Z_i$. Observe that an arc
  entering $W_i$ but not entering $X$ does not contribute to $f_{W}(Z)
  + \varrho_{E}(Z)$, thus $f_{W}(Z) + \varrho_{E}(Z) =\sum_{i=1}^t
  (f_{W_i}(Z_i) + \varrho_{E_i}(Z_i))=r_{W}(E)$.
\end{proof}

\subsection{Proof of Theorem \ref{thm:main}}

\begin{proof}[Proof of Theorem \ref{thm:main}]
By Claim \ref{cl:Z_1Z_2},
if $W \in \cL$ and $Z_1,Z_2$ are nonempty disjoint subsets of $W$, then there is a transversal of size
$f_W(Z_1)+f_W(Z_2)-k+1$, thus $\gamma$ is at most the minimum on the right hand side of \eqref{eq:gamma} (this is true even if $\gamma < k$).

To show that equality holds for some $W \in \cL$,
let $H$ be a minimum transversal,
and let $D'=D^+-H$. By Theorem
\ref{thm:minarcset}, there exists $W \in \cL$ and an
$\cL[W]$-compatible subpartition $\cX$ of $W$ such that $H\subseteq D[W]$ and $\sum_{X \in
  \cX} r_{D', W}^\oplus (\delta^{in}_{D'[W]}(X))= k(|\cX|-1)-1$. Let us choose a minimum transversal $H$ for which
  $W$ is the smallest possible, and (subject to that) $\cX$ has the smallest possible cardinality;
this implies that $r_{D', W}^\oplus (\delta^{in}_{D'[W]}(X))<k$ for every $X
\in \cX$.
\begin{claim}
$|\cX|=2$. \label{cl:exchange}
\end{claim}
\begin{proof}
Suppose for contradiction that $|\cX| \geq 3$. Then $0<r_{D',
  W}^\oplus (\delta^{in}_{D'[W]}(X))<k$ for every $X \in \cX$;
furthermore, by the assumption $\gamma \geq k$ and Corollary
\ref{cor:gamma<k}, all of these ranks except for at most one were
originally $k$ in $D$. Let $X_0$ be one of the members of \cX\ for
which $r_{D, W}^\oplus (\delta^{in}_{D[W]}(X_0))= k$, and let $X_1$ be
another member.  Let $E_1= \delta^{in}_{D'[W]}(X_1)$, and consider the
following arc exchange operation.

\paragraph{$(X_0,E_1)$-{\sc exchange}} By Lemma \ref{lem:decrease_rank}, there exists an arc $e\in D'[W]$ such that
$r_{D'-e, W}^\oplus (E_1-e)<r_{D', W}^\oplus (E_1)$.
Choose an arbitrary arc $e_0 \in H$ whose head is in $X_0$
(such an arc exists because $r_{D', W}^\oplus (\delta^{in}_{D'[W]}(X_0))<r_{D, W}^\oplus (\delta^{in}_{D[W]}(X_0))$).
Let $H_1=H-e_0+e$.
\bigskip

By the choice of $H$, there is no $W'
\subset W$ such that $H_1$ is a transversal of $\cL[W']$-tight \karb s in $D'_{W'}$.
By the choice of $e$, $H_1$ is still a transversal of $\cL[W]$-tight \karb s, so it
is a minimum transversal. We can apply the exchange operation repeatedly until we
obtain a minimum transversal $H''$ for which $ r_{D'', W}^\oplus (\delta^{in}_{D''[W]}(X_0))= k$, where $D''=D^+-H''$.
At this
point, $\cX-X_0$ is a good subpartition for $H''$ that has fewer
members than $\cX$, in contradiction to the choice of $H$ and $\cX$.
\end{proof}

We obtained that $\cX$ is a subpartition consisting of two members, so $\cX= \{X_1, X_2\}$ and $r_{D', W}^\oplus
(\delta^{in}_{D'[W]}(X_1))+r_{D', W}^\oplus (\delta^{in}_{D'[W]}(X_2))=
k-1$. The next claim shows that $H$ can be modified so that the arc sets in the formula become \elem.

\begin{claim} \label{cl:Hstar}
There is a minimum transversal $H^*$ of $\cL[W]$-tight \karb s such that $\delta^{in}_{D^*[W]}(X_j)$ is $(D^*,X_j)$-\elem\
and $r_{D^*, W}^\oplus (\delta^{in}_{D^*[W]}(X_j))=r_{D', W}^\oplus (\delta^{in}_{D'[W]}(X_j))$
for $j=1,2$ (where $D^*$ denotes $D^+-H^*$).
\end{claim}
\begin{proof}
If $\delta^{in}_{D'[W]}(X_j)$ is $(D',X_j)$-\elem\ for $j=1,2$, then $H$ has the required properties.
Suppose that $\delta^{in}_{D'[W]}(X_j)$ is not $(D',X_j)$-\elem\ . This means that if we recursively compute the rank
of $\delta^{in}_{D'[W]}(X_j)$, then at some point we have to compute a rank $r_{D',W'}(E')$
for some $W'\in \cL[W]-W$ and some $E'\subseteq \delta_{D'}^{in}(W')$, but the smallest $\cL[W']$-compatible
subpartition $\cY$ that determines $r_{D',W'}(E')$ has at least two members.

Since $0<r_{D',W'}(E')<k$, we have $0<r_{D',
  W'}^\oplus(\delta^{in}_{E'\cup D'[W']}(Y))<k$ for every $Y \subseteq
\cY$. Let $E=E'\cup (H \cap \delta_{D}^{in}(W'))$. By the assumption
$\gamma \geq k$ and Corollary \ref{cor:gamma<k}, we know that
$r_{W'}^\oplus(\delta^{in}_{E \cup D[W']}(Y))=k$ for all but at most
one member of $\cY$; let $Y_0$ be a member for which it is $k$, and
let $Y_1$ be another member.  Let $E_1=\delta^{in}_{E' \cup
  D'[W']}(Y_1)$.  By the same argument as in the proof of Claim
\ref{cl:exchange}, a $(Y_0,E_1)$-{\sc exchange} operation results in a
transversal of the same size as $H$, for which
$r_{D',W'}^\oplus(\delta^{in}_{E' \cup D[W']}(Y_0))$ increases by one.
By applying the exchange operation repeatedly, we eventually obtain a
transversal $H''$ such that $|H''|=|H|$ and $r_{D'',W'}^\oplus
(\delta^{in}_{E'' \cup D''[W']}(Y_0))= k$, where $E''= E - H''$ and
$D''=D^+-H''$. At this point, $\cY-Y_0$ also determines the rank
$r_{D'',W'}(E'') = r_{D',W'}(E')$, and has fewer members than $\cY$.

By repeating this procedure, we eventually obtain a transversal $H^*$ which satisfies the claimed properties.
\end{proof}

Let $H^*$ be the minimum transversal given by Claim \ref{cl:Hstar}. By Lemma \ref{lem:elem2}, there is a nonempty set
$Z_j \subseteq X_j$ such that $r_{D^*, W}^\oplus
(\delta^{in}_{D^*[W]}(X_j))=f_{D^*,W}(Z_j)$, for both $j=1,2$. Thus $f_{D^*,W}(Z_1)+f_{D^*,W}(Z_2)=k-1$.
Since the removal of an arc from $D$ can decrease
$f_{W}(Z_1)+f_{W}(Z_2)$ by at most one, we have $\gamma=|H^*| \geq f_{W}(Z_1)+f_{W}(Z_2)-k+1$. As the reverse
inequality has already been proved, this completes the proof of the theorem.
\end{proof}

\fi

\end{document}

%% file: pszkod.ltx
\makeatletter

\newcounter{Kulso}

\newlength{\tabhossz}
\newcommand{\tab}{\par
\advance\@totalleftmargin \tabhossz
\advance\linewidth -\tabhossz
\parshape \@ne \@totalleftmargin \linewidth
\addtolength{\labelsep}{\tabhossz}%
}
\newcommand{\untab}{\par
\advance\@totalleftmargin -\tabhossz   
\advance\linewidth \tabhossz
\parshape \@ne \@totalleftmargin \linewidth
\addtolength{\labelsep}{-\tabhossz}%
}

\newcommand{\psz@cimke}[1]{\item[{\makebox[\labelwidth][l]{#1}}]}

\makeatother

%% file: rv.pdf_t
\begin{picture}(0,0)%
\includegraphics{rv.pdf}%
\end{picture}%
\setlength{\unitlength}{4144sp}%
\begingroup\makeatletter\ifx\SetFigFont\undefined%
\gdef\SetFigFont#1#2#3#4#5{%
  \reset@font\fontsize{#1}{#2pt}%
  \fontfamily{#3}\fontseries{#4}\fontshape{#5}%
  \selectfont}%
\fi\endgroup%
\begin{picture}(2384,3092)(3229,-4121)
\put(3601,-1411){\makebox(0,0)[lb]{\smash{{\SetFigFont{14}{16.8}{\familydefault}{\mddefault}{\updefault}{\color[rgb]{0,0,0}$X_1$}%
}}}}
\put(4096,-2401){\makebox(0,0)[lb]{\smash{{\SetFigFont{14}{16.8}{\familydefault}{\mddefault}{\updefault}{\color[rgb]{0,0,0}$v$}%
}}}}
\put(4096,-4066){\makebox(0,0)[lb]{\smash{{\SetFigFont{14}{16.8}{\familydefault}{\mddefault}{\updefault}{\color[rgb]{0,0,0}$s$}%
}}}}
\put(5086,-1411){\makebox(0,0)[lb]{\smash{{\SetFigFont{14}{16.8}{\familydefault}{\mddefault}{\updefault}{\color[rgb]{0,0,0}$X_2$}%
}}}}
\end{picture}%

%% file: mand_constr.pdf_t
\begin{picture}(0,0)%
\includegraphics{mand_constr.pdf}%
\end{picture}%
\setlength{\unitlength}{4144sp}%
\begingroup\makeatletter\ifx\SetFigFont\undefined%
\gdef\SetFigFont#1#2#3#4#5{%
  \reset@font\fontsize{#1}{#2pt}%
  \fontfamily{#3}\fontseries{#4}\fontshape{#5}%
  \selectfont}%
\fi\endgroup%
\begin{picture}(4727,2014)(4171,-5426)
\put(4231,-5371){\makebox(0,0)[lb]{\smash{{\SetFigFont{14}{16.8}{\familydefault}{\mddefault}{\updefault}{\color[rgb]{0,0,0}$u$}%
}}}}
\put(5941,-3571){\makebox(0,0)[lb]{\smash{{\SetFigFont{14}{16.8}{\familydefault}{\mddefault}{\updefault}{\color[rgb]{0,0,0}$v$}%
}}}}
\put(6931,-5371){\makebox(0,0)[lb]{\smash{{\SetFigFont{14}{16.8}{\familydefault}{\mddefault}{\updefault}{\color[rgb]{0,0,0}$u$}%
}}}}
\put(4591,-4201){\makebox(0,0)[lb]{\smash{{\SetFigFont{14}{16.8}{\familydefault}{\mddefault}{\updefault}{\color[rgb]{0,0,0}$a$}%
}}}}
\put(8641,-3618){\makebox(0,0)[lb]{\smash{{\SetFigFont{14}{16.8}{\familydefault}{\mddefault}{\updefault}{\color[rgb]{0,0,0}$v$}%
}}}}
\put(7328,-4204){\makebox(0,0)[lb]{\smash{{\SetFigFont{14}{16.8}{\familydefault}{\mddefault}{\updefault}{\color[rgb]{0,0,0}$x_a$}%
}}}}
\end{picture}%

%% file: TKpl.pdf_t
\begin{picture}(0,0)%
\includegraphics{TKpl.pdf}%
\end{picture}%
\setlength{\unitlength}{4144sp}%
\begingroup\makeatletter\ifx\SetFigFont\undefined%
\gdef\SetFigFont#1#2#3#4#5{%
  \reset@font\fontsize{#1}{#2pt}%
  \fontfamily{#3}\fontseries{#4}\fontshape{#5}%
  \selectfont}%
\fi\endgroup%
\begin{picture}(3624,2824)(8343,-3746)
\put(10055,-2234){\makebox(0,0)[lb]{\smash{{\SetFigFont{12}{10.8}{\familydefault}{\mddefault}{\updefault}{\color[rgb]{0,0,0}$s$}%
}}}}
\put(10753,-2234){\makebox(0,0)[lb]{\smash{{\SetFigFont{12}{10.8}{\familydefault}{\mddefault}{\updefault}{\color[rgb]{0,0,0}$v$}%
}}}}
\put(10853,-1495){\makebox(0,0)[lb]{\smash{{\SetFigFont{12}{10.8}{\familydefault}{\mddefault}{\updefault}{\color[rgb]{0,0,0}$W$}%
}}}}
\put(11512,-1854){\makebox(0,0)[lb]{\smash{{\SetFigFont{12}{10.8}{\familydefault}{\mddefault}{\updefault}{\color[rgb]{0,0,0}$x_1$}%
}}}}
\put(11492,-2832){\makebox(0,0)[lb]{\smash{{\SetFigFont{12}{10.8}{\familydefault}{\mddefault}{\updefault}{\color[rgb]{0,0,0}$x_2$}%
}}}}
\end{picture}%